\documentclass[11pt]{amsart}   	
\usepackage{geometry}
\usepackage[headings]{fullpage}
\usepackage{setspace}
\usepackage{colonequals}
\usepackage{tikz}
\usetikzlibrary{arrows.meta, positioning}

\newtheorem{innercustomthm}{Theorem}
\newenvironment{customthm}[1]
  {\renewcommand\theinnercustomthm{#1}\innercustomthm}
  {\endinnercustomthm}

    \newtheorem{innercustomconj}{Conjecture}
\newenvironment{customconj}[1]
  {\renewcommand\theinnercustomconj{#1}\innercustomconj}
  {\endinnercustomconj}

\title{The many faces of a logarithmic scheme}
\author{Thibault Poiret {\it \&} Dhruv Ranganathan}

\usepackage{amsmath, amssymb, mathrsfs, amsthm, shorttoc, stmaryrd, tikz-cd, graphicx}
\usepackage{mathtools} 
\usepackage{comment}
\usepackage[backref=page]{hyperref}
\hypersetup{
  colorlinks   = true,          
  urlcolor     = blue,          
  linkcolor    = violet,          
  citecolor   = blue             
}
\usepackage[all]{xy}

\makeatletter
\providecommand{\leftsquigarrow}{%
  \mathrel{\mathpalette\reflect@squig\relax}%
}
\newcommand{\reflect@squig}[2]{%
  \reflectbox{$\m@th#1\rightsquigarrow$}%
}
\makeatother

\let\oldtocsection=\tocsection
 
\let\oldtocsubsection=\tocsubsection
 
\let\oldtocsubsubsection=\tocsubsubsection
 
\renewcommand{\tocsection}[2]{\hspace{0em}\oldtocsection{#1}{#2}}
\renewcommand{\tocsubsection}[2]{\hspace{1em}\oldtocsubsection{#1}{#2}}
\renewcommand{\tocsubsubsection}[2]{\hspace{2em}\oldtocsubsubsection{#1}{#2}}

\usepackage{enumitem}

\usepackage[capitalize]{cleveref}
\crefformat{equation}{(#2#1#3)}

\AtBeginDocument{\renewcommand{\ref}[1]{\cref{#1}}}
\numberwithin{equation}{subsection}

\usepackage{pgf,tikz}
\usetikzlibrary{matrix, calc, arrows}
\usepackage{tikz-cd} 
\usetikzlibrary{decorations.pathmorphing,shapes}

\newcommand{\LSch}{\cat{LSch}}
\newcommand{\LSchint}{\cat{LSch}^{\on{int}}}

\newcommand{\LSchfs}{\cat{LSch}^{\on{fs}}}
\newcommand{\LSchpointed}{\cat{LSch}^{\on{pointed}}}

\newcommand{\ini}{\operatorname{in}}

\makeatletter
\newcommand*{\doublerightarrow}[2]{\mathrel{
  \settowidth{\@tempdima}{$\scriptstyle#1$}
  \settowidth{\@tempdimb}{$\scriptstyle#2$}
  \ifdim\@tempdimb>\@tempdima \@tempdima=\@tempdimb\fi
  \mathop{\vcenter{
    \offinterlineskip\ialign{\hbox to\dimexpr\@tempdima+1em{##}\cr
    \rightarrowfill\cr\noalign{\kern.5ex}
    \rightarrowfill\cr}}}\limits^{\!#1}_{\!#2}}}
\newcommand*{\triplerightarrow}[1]{\mathrel{
  \settowidth{\@tempdima}{$\scriptstyle#1$}
  \mathop{\vcenter{
    \offinterlineskip\ialign{\hbox to\dimexpr\@tempdima+1em{##}\cr
    \rightarrowfill\cr\noalign{\kern.5ex}
    \rightarrowfill\cr\noalign{\kern.5ex}
    \rightarrowfill\cr}}}\limits^{\!#1}}}
\makeatother

\newcommand{\on}[1]{\operatorname{#1}}
\newcommand{\bb}[1]{{\mathbb{#1}}}

\newcommand{\ca}[1]{{\mathcal{#1}}}
\newcommand{\bd}[1]{{\mathbf{#1}}}

\newcommand{\ul}[1]{{\underline{#1}}}

\newcommand{\cat}[1]{\bd{#1}}

\newcommand{\lra}{\longrightarrow}
\newcommand{\hra}{\hookrightarrow}

\newcommand{\iso}{\stackrel{\sim}{\lra}}

\def\sat{{\rm sat}}
\def\:{\colon}
\def\.{,\dots,}

\def\CC{\mathbb C}
\def\ZZ{\mathbb Z}

\def\NN{\mathbb N}

\def\o#1{\overline{#1}}
\def\gp{\textsf{gp}}
\def\int{\textsf{int}}
\def\sat{\textsf{sat}}
\def\gp{\textsf{gp}}
\def\tor{\textsf{tor}}

\newcommand{\Hom}{\operatorname{Hom}}
\newcommand{\Spec}{\operatorname{Spec}}

\theoremstyle{definition}
\newtheorem{definition}{Definition}[section]

\newtheorem{situation}[definition]{Situation}

\newtheorem{algorithm}[definition]{Algorithm}

\theoremstyle{plain}

\newtheorem{proposition}[definition]{Proposition}
\newtheorem{lemma}[definition]{Lemma}
\newtheorem{theorem}[definition]{Theorem}
\newtheorem{corollary}[definition]{Corollary}

\theoremstyle{remark}

\newtheorem{remark}[definition]{Remark}

\newtheorem{example}[definition]{Example}

  \makeatletter
\def\@tocline#1#2#3#4#5#6#7{\relax
  \ifnum #1>\c@tocdepth 
  \else
    \par \addpenalty\@secpenalty\addvspace{#2}%
    \begingroup \hyphenpenalty\@M
    \@ifempty{#4}{%
      \@tempdima\csname r@tocindent\number#1\endcsname\relax
    }{%
      \@tempdima#4\relax
    }%
    \parindent\z@ \leftskip#3\relax \advance\leftskip\@tempdima\relax
    \rightskip\@pnumwidth plus4em \parfillskip-\@pnumwidth
    #5\leavevmode\hskip-\@tempdima
      \ifcase #1
       \or\or \hskip 1em \or \hskip 2em \else \hskip 3em \fi%
      #6\nobreak\relax
    \dotfill\hbox to\@pnumwidth{\@tocpagenum{#7}}\par
    \nobreak
    \endgroup
  \fi}
\makeatother

\renewcommand{\phi}{\varphi}

\newcounter{nootje}
\begin{document}

\begin{abstract} 
A standard assumption in the study of logarithmic structures is ``fineness'', but this assumption is not preserved by intersections, fiber products, and more general limits. We explain how a coherent logarithmic scheme $X$ has a natural decomposition into ``fine faces'' -- a collection of fine logarithmic subschemes. This allows for importation of the techniques of tropical geometry into the study of more general logarithmic schemes. We explain the relevance of the construction in enumerative geometry and moduli via a range of examples. Techniques developed in the study of binomial schemes lead to effective algorithms for computing the fine faces. 
\end{abstract}

\maketitle
	
\setcounter{tocdepth}{1}
\tableofcontents


\section{Introduction}

A {\it logarithmic structure} on a scheme $X$ equips it with a collection of {\it monomial functions}, encoded by a sheaf of monoids. The canonical example is a toric variety $V$, where we declare the functions which vanish at most on the toric boundary to be monomial. Logarithmic structures on schemes are often obtained by pulling back monomial functions along maps
\[
X\to V
\]
to toric varieties. Logarithmic structures that are locally obtained from such maps play a distinguished role in the theory, and are termed {\it fine}. If the toric varieties are required to be normal, these log structures are also called {\it fine and saturated} or {\it fs}. Much of the literature in logarithmic geometry assumes fineness and saturatedness. In particular, techniques coming from toric geometry, such as cone complexes and tropical moduli theory are only accessible in the fine setting. It also appears to be the correct hypothesis for moduli problems. For these reasons, general coherent logarithmic structures have largely been neglected. Nevertheless, they arise frequently in nature. The simple reason for this is that fineness is not preserved by intersections, fiber products, and more general limits. 

Our main result is that coherent logarithmic schemes can be ``decomposed'' into fine ones. In \ref{section: faces} we introduce the {\it fine faces} of a coherent logarithmic scheme. Each is a fine logarithmic scheme. 

\begin{customthm}{A}
Let $X$ be a coherent logarithmic scheme and let $\mathcal F$ be its set of fine faces. The collection of underlying schemes admit canonical, jointly surjective, closed immersions:
\[
\{\underline F_i\to \underline X\}_{F\in \mathcal F}.
\]
One of the fine faces is the integralization of $X$.
\end{customthm}

A simple example captures the essence of the result. Let $X$, $Y$, and $B$ be toric varieties and suppose we have equivariant maps $X\to B$ and $Y\to B$. Equip the data with the toric logarithmic structures. The fiber product $X\times_B Y$ is not necessarily toric, however, the reduced underlying scheme of every component is toric. After re-indexing to remove some redundancies, the reduced fine faces are these toric components. Further details on this can be found in \ref{subsection:toric_fiber_products}. 

Faces of logarithmic schemes arise naturally whenever fiber products of logarithmic schemes arise, and this is particicularly common in logarithmic Gromov--Witten theory. In this context, the scheme theoretic fiber product is typically easier to compute, while the fine logarithmic fiber product is what one wants to compute. In \ref{section: examples}, we explain some of these occurrences, including toric fiber products in~\ref{section:toric_fibreprod} and toric Chow quotients in~\ref{sec: Chow-quotients}. This motivates the computation of the faces of a coherent logarithmic scheme. 

In \ref{section:algorithms_on_monoids}, we present various algorithms which compute algebro-geometric and logarithmic constructions. Using these, including a {\it monoidal Buchberger algorithm}, we show the following:

\begin{customthm}{B}
    The algorithms of \ref{section:algorithms_on_monoids} compute the fine faces of a coherent logarithmic scheme.
\end{customthm}

The algorithm depends on relatively detailed knowledge of the logarithmic structure, such as the explicit knowledge of charts. In moduli contexts, this is provided by tropical geometry~\cite{cavalieri2017moduli}. 

The fine faces of a logarithmic scheme $X$ admit fine logarithmic structures. We may saturate them to obtain fs faces. Saturation is analogous to normalization, and is surjective at the level of underlying schemes. The upshot is that the log structures on fs faces locally come from normal toric varieties, hence can be described in terms of cone complexes. To understand the cone complex associated to a face, we can use the ``extended cone complex'' construction of~\cite{AbramovichCaporasoPayne}. We record the following formal consequence of the construction. 

\begin{customthm}{C}
    Let $F$ be a fs face of a fiber product of fs logarithmic schemes $X\times_B Y$. The cone complex of $F$ is a face of the fiber product of the extended cone complexes associated to $X$ and $Y$ over that of $B$.
\end{customthm}

An analogous result holds the faces of any coherent logarithmic scheme. 

\subsection{Faces of moduli of abelian schemes} The paper focuses on schemes that are equipped with logarithmic structures, and extracts their ``fine pieces''. There are natural moduli spaces in higher dimensional geometry that we believe should fit into our framework. 

In~\cite{AlexeevCompleteModuli} Alexeev constructs a complete moduli space of varieties with semiabelian group actions, and in particular, complete moduli of polarized abelian varieties and polarized toric varieties. It is noted there that the moduli space $\overline{\mathcal A}_g$ of stable pairs compactifying the space $\mathcal A_g$ of principally polarized abelian varieties of dimension $g$ is {\it reducible}. In later work, Olsson identified the main component, the closure of $\mathcal A_g$, with the locus of semi-abelic pairs admitting a fine logarithmic structure~\cite{OlssonAbelianBook}.

\begin{customconj}{D}
    The moduli space $\overline{\mathcal A}_g$ admits a coherent logarithmic structure. The maximal fine faces of this logarithmic structure coincide with the components of the normalization of $\overline{\mathcal A}_g$. 
\end{customconj}

The evidence for this conjecture is that the ``other components'' of Alexeev's space appear to have natural stratifications indexed by tropical moduli data.

We expect that, if this statement is true, it would lend a logarithmic moduli interpretation to {\it all} the components of the Alexeev moduli space, in the same way that Olsson endows the principal component with one. 

One can make an analogous conjecture for any moduli space of varieties with semiabelian group action, including polarized toric pairs, and in fact, it appears that a similar picture holds for moduli spaces of hyperplane arrangements~\cite{HKT06}. It seems potentially fruitful to use coherent non-fine logarithmic structures to understand at least some of the moduli spaces of KSBA type. 

The upshot of the paper is essentially that {\it coherent logarithmic structures are to fine ones, what binomial schemes are to toric varietes}. A number of techniques have been developed to study binomial schemes~\cite{EisenbudSturmfels}, and they can be ported over to this setting. 

\subsection{Related work} The {\it extended tropicalization} of a toroidal embedding has guided the results in this paper~\cite{AbramovichCaporasoPayne,Payne2009Analytification}. The basic observation in this paper is the fact that taking extended tropicalization does not commute with fiber products of cone complexes -- the extended tropicalization of a fiber product in fine logarithmic schemes is not the fiber product of the extended tropicalizations; the analogous statement {\it is} true for ordinary cone complexes. The explanation is that the fiber product of the extended tropicalization is larger, and sees the scheme theoretic fiber product of these schemes. 

From the logarithmic geometry literature, we rely on the notion of logarithmic structures based on {\it pointed monoids}, i.e. ones with absorbing elements, to construct the fine faces. We learned of this notion from work of Huszar--Marcus--Ulirsch~\cite{HMU} who studied it in the context of ``sublogarithmic'' morphisms. 

Finally, logarithmic Gromov--Witten theory has provided key examples and geometric motivation. Specifically, the phenomena that we have tried to isolate here arise in work related to product formulas~\cite{BattistellaNabijouRanganathan,NabijouRanganathan} and degeneration formulas~\cite{AbramovichChenGrossSiebert,RanganathanExpansions}.

\subsection*{Acknowledgements} We are grateful to N. Nabijou and S. Molcho for inspiring discussions in the context of the papers~\cite{NabijouRanganathan,MR21}. The authors were both supported by EPSRC New Investigator Award EP/V051830/1. T.P. is supported by EPSRC New Investigator Award EP/X002004/1 (PI: Yoav Len). D.R. is supported by EPSRC Frontier Research Grant EP/Y037162/1. 

\subsection*{Notation and conventions} All our rings are commutative with unity. A monoid is a set with a binary operation that is associative, commutative, and unital. Given two elements $a=(a_i)_{i\in I}$ and $b=(b_i)_{i\in I}$ in $\bb N^I$, we denote by $a\vee b$ the coordinatewise-maximum $(\operatorname{max}(a_i,b_i))_{i\in I}$.

\section{Logarithmic preliminaries}

We assume the reader is familiar with logarithmic structures, but we recall and highlight the aspects that are best kept in mind during the discussion. For a comprehensive introduction to log schemes, we refer to \cite{Kato1989Logarithmic-str} and \cite{Ogus2018Lectures}.

\subsection{Logarithmic structures and charts}

We recall the fundamental notion.

\begin{definition}
A \emph{prelog structure} on a scheme $\ul X$ is a morphism $\alpha\colon M_X \to \ca O_X$ of sheaves of monoids  on $\ul X$ in the \'etale topology, where $\ca O_X$ is given its structure of multiplicative monoid. We say $\alpha$ is a \emph{log structure} if $\alpha^{-1}\ca O_X^\times \to \ca O_X^\times$ is an isomorphism. 
\end{definition}

A prelog structure can be made universally into a log structure, and we will sometimes abuse notation and do this implicitly. Logarithmic structures can be pulled back and pushed forward along scheme theoretic maps. If $f\colon\ul Y \to \ul X$ is a map of schemes, these are denoted as $f^*$ and $f_*$. An important role in the theory is played by the {\it characteristic sheaf}:
\[
\overline M_X\colonequals M_X/\mathcal O_X^\star.
\]
Its role can be understood as follows. All units are always declared as monomials, and $\overline M_X$ keeps track of monomials modulo units, i.e. of the vanishing orders of monomials on every irreducible component of their vanishing loci.

Examples can be generated by starting with the monoid of nontrivial monomials, and building a scheme out of it.

\begin{example}\label{example:global_chart}
Let $P$ be a monoid, and write $\Spec \ZZ[P]$ for the quotient of the free ring $F$ over $P$ by the relations necessary to turn $P \to \ZZ[P]$ into a multiplicative monoid map. Then $S:=\Spec \ZZ[P]$ carries a natural log structure, namely the one corresponding to the prelog structure $P \to \ca O_S$. The characteristic sheaf detects the ``extra monomials'', so its value on an open $U$ is the quotient of $P$ by the functions that become invertible on $U$. The characteristic sheaf is constructible. 

Assuming that $P$ in the previous example is finitely generated, it can be presented as a quotient.
\[
\mathbb N^I\to P.
\]
As a result, $\Spec \ZZ[P]$ is a {\it binomial subscheme} of $\mathbb A^I_\ZZ$. If, in addition, the associated ideal is prime then $\Spec \ZZ[P]$ is a possibly non-normal affine toric variety. 
\end{example}

\begin{definition}
    The category of \emph{schemes with log structures} has objects the pairs $(\ul X, \alpha)$, where $\ul X$ is a scheme and $\alpha\colon M_X \to \ca O_X$ is a log structure on $\ul X$. We will often abusively write $(\ul X,M_X)$. A morphism $f\colon (\ul X,\alpha) \to (\ul Y,\beta)$ is the data of a morphism of schemes $\ul f\colon \ul X \to \ul Y$ and a morphism of log structures $f^*M_Y \to M_X$. We say that $f$ is \emph{strict} if $f^*M_Y \to M_X$ is an isomorphism. We use the underline notation for either a scheme, or the underlying scheme of a scheme with log structure.
\end{definition}

\begin{definition}
    A (global) \emph{chart} for a scheme with log structure $X$ is a strict morphism $X \to \Spec \ZZ[P]$ for some monoid $P$. We also say that $X$ has a \emph{chart by the monoid $P$}. We say that $X$ is a \emph{log scheme} (resp. \emph{Zariski log scheme}), when $X$ admits charts by finitely generated monoids \'etale-locally (resp. Zariski locally) on $\ul X$.\footnote{In the literature, finite-generatedness of the charting monoids is not always required in the definition of log schemes.} A \emph{morphism of log schemes} is a morphism of schemes with log structures between two log schemes. We write $\LSch$ for the category of log schemes.
\end{definition}

\begin{example}
The inclusion of the monomials in the structure sheaf gives a log structure on $\bb A^n$, making it a log scheme. Similarly but more generally, for any scheme $\ul X$ and any divisor $D$ on $\ul X$, one can define the \emph{divisorial log structure corresponding to $D$} as the inclusion in $\ca O_X$ of those functions which are invertible away from $D$. When $D$ has normal crossings, this produces a log scheme. If $X,Y$ are log schemes with divisorial log structures from divisors $D,D'$, then morphisms of log schemes $X \to Y$ are equivalent to morphisms of schemes $f \colon \ul X \to \ul Y$ together with a factorization $f^*D' \to D \to X$.
\end{example}

\begin{definition}\label{definition:stratification_by_char_monoid}
    If $x \to X$ is a geometric point, we call \emph{characteristic monoid} of $X$ at $x$ the \'etale stalk $\o M_{X,x}$ of the characteristic sheaf $\o M_X=M_X/\ca O_X^\star$. Some small enough strict \'etale neighbourhood of $x$ in $X$ can be charted by $\o M_{X,x}$.
\end{definition}

\subsection{Logarithmic adjectives: integral, fine, saturated, etc}

As we have said, various classes of logarithmic schemes play distinguished roles in the theory, and the topic of this paper involves the relationships between these. We now introduce these formally. 

Let $M$ be a monoid. The category of monoid maps from $M$ to groups has an initial object, called the \emph{groupification} of $M$ and denoted by $M \to M^{\sf gp}$.

\begin{definition}
    Let $M$ be a monoid.
    \begin{itemize}
        \item We say $M$ is \emph{sharp} if the unit group $M^\times$ is $\{0\}$.
        \item We say $M$ is \emph{integral} if $M \to M^{\sf gp}$ is injective.
        \item If $M$ is integral, we say it is \emph{saturated} if for all $n\in\NN\backslash \{0\}$ and $m\in M^{\sf gp}$ such that $nm$ is in $M$, we also have $m\in M$.
        \item We say $M$ is \emph{fine and saturated}, abbreviated \emph{\sf fs}, if it is integral, saturated and finitely generated.
    \end{itemize}
    The category of monoid maps $M \to N$ where $N$ is sharp (resp. integral, resp. saturated) has an initial object $M \to \o M$ (resp. $M \to M^{\sf int}$, resp. $M \to M^{\sf sat}$). We call it the \emph{sharpening}, resp. \emph{integralization}, resp. \emph{integral saturation} of $M$. More explicitly, $\o M$ is the quotient of $M$ by its units, $M^{\sf int}$ is the image of $M$ in $M^{\sf gp}$, and $M^{\sf sat}$ is the intersection of $M^{\sf gp}$ and $\bb Q M^{\sf int}$ in $\bb Q M^{\sf gp}$. We have natural factorizations
    \[
    M \twoheadrightarrow M^{\sf int} \hra M^{\sf sat} \hra M^{\sf gp}.
    \]
    We extend these definitions, notations and remarks to sheaves of monoids in the obvious way.
\end{definition}

To clarify the definitions, consider a monoid $M$ and the associated scheme $\Spec \ZZ[M]$. We have already explained it is a binomial scheme. The condition that $M$ is sharp corresponds to this scheme not having diagonalizable groups as factors. If $M$ is sharp and integral, the scheme is a toric variety with dense torus $\on{Spec} \ZZ[M^{\sf gp}]$. If $M$ is sharp and saturated, this toric variety is normal.

\begin{definition}
    Let $X$ be a log scheme. We say that $X$ is \emph{integral} (resp. \emph{fine and saturated} or \emph{fs}) if it admits \'etale-local charts by monoids which are integral (resp. fs). We write $\LSchint$, resp. $\LSchfs$ for the full subcategory of $\LSch$ whose objects are integral, resp. fs log schemes.
\end{definition}

\begin{proposition}
    The inclusions $\LSchint \to \LSch$ and $\LSchfs \to \LSch$ admit right adjoints, which we respectively call \emph{integralization} and \emph{fine saturation}. We write $X^{\sf int}$ for the integralization of a log scheme $X$, and $X^{\sf fs}$ for its fine saturation.
\end{proposition}

\begin{proof}
    Let $X$ be a log scheme. If $X$ has a global chart $f \colon X \to \Spec \ZZ[P]$ by a finitely generated monoid $P$, define $X^{\sf int}$ as the fiber product of $f$ with $\Spec \ZZ[P^{\sf int}] \to \Spec \ZZ[P]$. The resulting $X^{\sf int}$ does not depend on the choice of chart, so we may extend this construction to log schemes with only local charts and get a functor $\LSch \to \LSchint$. This functor is right adjoint to the forgetful functor $\LSchint \to \LSch$ since the integralization functor on monoids is left adjoint to the forgetful functor. Fine saturation is similar and we omit it.
\end{proof}

\begin{remark}\label{remark:fiber_prod_vs_fs_fiber_prod}
The ``underlying scheme" functor $\LSch \to \cat{Sch}$ commutes with limits, but $\LSchfs \to \cat{Sch}$ does not. Limits in $\LSch$ are easy to construct: compute the limit of the underlying schemes in $\cat{Sch}$, then endow the resulting scheme with the appropriate colimit of sheaves of monoids. A limit in $\LSchfs$ can be obtained by computing the limit in $\LSch$, then applying the fine saturation functor. Since log geometers almost always work in $\LSchfs$, it is natural to want to describe explicitly the fine saturation functor, and how it acts on underlying schemes.
\end{remark}

\subsection{Presentations of fiber products}

Let $f \colon X \to Y$ be a morphism of log schemes, $y$ a geometric point of $\ul Y$ and $x$ a geometric point of $\ul X_y$. Restrict to small neighbourhoods of $x$ and $y$ so that $X,Y$ are globally charted by $Q:=\overline M_{X,x}$ and $P:=\overline M_{Y,y}$ respectively. There may not exist a morphism $P \to Q$ making the following diagram commute in $\cat{LSch}$.

    \begin{equation}\label{eqn:commutative_diag_charting_morphisms}
        \begin{tikzcd}
            X\arrow{r}\arrow{d} & \Spec \bb Z[Q] \arrow{d}\\
            Y\arrow{r} & \Spec \bb Z[P].
        \end{tikzcd} 
    \end{equation}
Indeed, given $p,q\in P$, the images of $p,q$ under $P \to M_Y(Y) \to M_X(X)$ may differ by some nontrivial $\lambda\in\ca O_{X}^\times(X)$. The natural map $P \to Q$ would then identify $p$ and $q$, thereby ``forgetting the nontriviality of the unit $\lambda$". This motivates the following definition.

\begin{definition}
    Let $f \colon X \to Y$ be a morphism of log schemes. A \emph{global chart} for $f$ is a commutative diagram \ref{eqn:commutative_diag_charting_morphisms} where the horizontal arrows are strict and $P,Q$ are finitely generated monoids.
\end{definition}

For $x$ a point in a logarithmic scheme $X$, the characteristic monoid $Q:=\overline M_{X,x}$ is a ``lower bound'' for the charting monoid near $x$, but we can produce other charting monoids by allowing them to have units. For example, $X:=\mathbb G^k_m$ with the trivial log structure admits a strict morphism to $\Spec \ZZ[\ZZ^k]$, where the $i$-th generator of $\bb Z^k$ maps to the $i$-th coordinate parameter of $X$. This flexibility is used as follows: 

\begin{lemma}
    Any morphism $f\colon X \to Y$ in $\cat{LSch}$ has local charts, i.e. for any geometric points $y \to \ul Y$ and $x \to \ul X_y$ there exist strict \'etale neighbourhoods $x \to U \to X$ and $y \to V \to Y$ such that $U \to V$ has a global chart.
\end{lemma}

\begin{proof}
    Shrinking $X$ and $Y$ we may assume that there are global charts of $X,Y$ by $Q^0:=\o M_{X,x}$ and $P:=\o M_{Y,y}$ respectively. These charts correspond to sections of the quotients $M_{Y,y} \to P$ and $\o M_{X,x} \to Q^0$. The diagram
    \begin{equation}
        \begin{tikzcd}
            P \arrow[r] \arrow[d] & Q^0 \arrow{d}\\
            M_{Y,y} \arrow[r] & M_{X,x}
        \end{tikzcd} 
    \end{equation}
    may not commute, but it commutes up to units. Thus, there is a finitely generated subgroup $U\subset \ca O_{X,x}^\times$ such that $\on{Im}(P \to M_{Y,y} \to M_{X,x}) + U$ contains $\on{Im}(Q^0 \to M_{X,x})$ as submonoids of $M_{X,x}$. The global chart of $X$ by $Q^0$ induces one by $Q^0 \oplus U$. By definition of $U$, the composite $P \to M_{Y,y} \to M_{X,x}$ factors through $Q \to M_{X,x}$, which proves the lemma.
\end{proof}

\begin{remark}
    As a consequence of the proof, we see that any morphism of log schemes $X \to Y$ has local charts \ref{eqn:commutative_diag_charting_morphisms} where $P$ is a stalk of $M_Y$ and $Q$ is the direct sum of a stalk of $M_X$ with a finitely generated abelian group. In other words, we can chart after adding torus factors. 
\end{remark}

We can now expand on \ref{remark:fiber_prod_vs_fs_fiber_prod}. Let

\begin{equation}
        \begin{tikzcd}
            & Y \arrow{d}\\
            Z\arrow{r} & X
        \end{tikzcd} 
    \end{equation}
be a diagram of fs log schemes with a global chart
\begin{equation}
        \begin{tikzcd}
            & \Spec \bb Z[Q] \arrow{d}\\
            \Spec \bb Z[R]\arrow{r} & \Spec \bb Z[P].
        \end{tikzcd} 
    \end{equation}

Then, the fiber product $Y\times_X Z$ in $\cat{LSch}$ has a global chart by the pushout $M:=Q\oplus_P R$, which may not be fs (any finitely presented monoid arises as a colimit of \emph{free} ones). The fiber product in $\LSchfs$ has a global chart by $M^{\sf fs}$. Presentations for $P,Q,R$ naturally yield a presentation for $M$, so all we have to do to compute the fiber product in $\LSchfs$ is find an algorithm to compute $M^{\sf fs}$ from a presentation of $M$.

\section{Faces and strata of logarithmic schemes}

We begin by stating the objective. Let $X$ be a coherent logarithmic scheme. The integralization $X^{\sf int}$ is a fine logarithmic scheme that comes equipped with a closed immersion 
\[
X^{\sf int}\hookrightarrow X.
\]
This is the best approximation to $X$ in fine logarithmic schemes, but it loses information about $X$. We would like to ``stratify'' $X$ by its failure to be fine, such that each stratum is fine. Moreover, we would like these to form a collection that jointly surject onto $X$, thereby recovering some of the lost information about $X$. These will precisely be the fine faces of $X$.

\subsection{Pointed monoids, prime ideals} We use the notion of a pointed monoid, which has made prior appearances in logarithmic geometry~\cite{HMU}. 

\begin{definition}
    Let $\cat{Mon^0}$ be the category of pointed monoids where the distinguished element, which we will usually call $\infty$, is absorbing. A morphism of pointed monoids is a morphism of the underlying monoids which sends $\infty$ to $\infty$. There is a forgetful functor $\on{log}\colon \cat{Ring} \to \cat{Mon^0}$ mapping a ring to its underlying multiplicative pointed monoid, sending $0$ to $\infty$. 
\end{definition}

\begin{example}
    A pointed monoid $M$ is a singleton if and only if $0=\infty$. There exists only one such $M$ up to a unique isomorphism, the \emph{zero pointed monoid}.
\end{example}

\begin{lemma}\label{lemma:ring_to_mon0_to_mon_have_left_adjoints}
    The forgetful functors $\cat{Ring} \to \cat{Mon}^0 \to \cat{Mon}$ admit left adjoints. The left adjoint $\cat{Mon} \to \cat{Mon}^0$ formally adjoins the absorbing distinguished point, taking $M$ to $M \cup \{\infty\}$. We will (suggestively) write $M \mapsto \bb Z[M]$ for both the left adjoint $\cat{Mon}^0 \to \cat{Ring}$ and the composite $\cat{Mon} \to \cat{Ring}$. 
\end{lemma}

\begin{proof}
    For any monoid $M$ and pointed monoid $N$, we have
    \[
    \on{Hom}_{\cat{Mon}}(M,N)=\on{Hom}_{\cat{Mon}^0}(M\cup\{\infty\},N)
    \]
    since maps in $\cat{Mon}^0$ must send $\infty$ to $\infty$. The left adjoint $\cat{Mon} \to \cat{Ring}$ is the classical ``free ring over a monoid" construction, and $\cat{Mon}^0 \to \cat{Ring}$ is obtained analogously ($\infty$ must map to $0$).
\end{proof}

\begin{definition}\label{definition:rays_faces_etc_of_monoids}
    Let $P$ be a monoid. A \emph{face} of $P$ is a submonoid $Q$ such that any two elements of $P$ whose sum lies in $Q$ lie in $Q$. An \emph{ideal} of $P$ is a subset $I$ such that $I+P\subset I$. The ideal $I$ is \emph{proper} if $I\neq P$; and \emph{prime} if it is proper and for any $a,b$ in $P$ with $a+b\in I$ we have either $a\in I$ or $b\in I$. We define faces, ideals and primes of pointed monoids analogously, with the convention that ideals of pointed monoids must always contain $\infty$. We emphasize that faces of a pointed monoid are sub-pointed monoids, so they must contain $\infty$ as well.
\end{definition}

\begin{remark}\label{remark:faces_primes_of_monoid_are_faces_primes_of_pointification}
A subset of a monoid is a prime ideal if and only if its complement is a face. The faces of a pointed monoid $P$ are in bijection with its prime ideals via $F \mapsto (P \setminus F) \cup \{\infty\}$. If $P$ is a monoid, then faces, ideals and primes of $P$ are in natural bijection with faces, ideals and primes of $P\cup\{\infty\}$ via the map taking a subset $Q$ to $Q\cup\{\infty\}$.
\end{remark}

\begin{definition}\label{definition:quotients_of_pointed_monoids_by_ideals}
    Let $P$ be a pointed monoid and $I\subset P$ an ideal. We call \emph{quotient of $P$ by $I$} and denote by $P/I$ the set-theoretic quotient of $P$ by the relations identifying all elements of $I$ and nothing else.\footnote{We warn the reader that this is not to be confused with the quotient of a monoid by a subgroup. Hopefully, whichever we mean will always be clear from context.} We endow $P/I$ with its pointed monoid structure inherited from $P$ (which exists since $I$ is an ideal). Observe that if we denote by $F$ the face $P\backslash I \cup \{\infty\}$, then the natural map $F \to P/I$ is an isomorphism.
\end{definition}

\begin{lemma}\label{lemma:quotients_of_pointed_monoids_by_ideals_commute_with_ringification}
    Let $P$ be a pointed monoid and $I\subset P$ an ideal. Then $\bb Z[P/I]=\bb Z[P]/I\bb Z[P]$.
\end{lemma}
\begin{proof}
    Both sides are initial among rings over $\bb Z[P]$ in which the image of $I$ is zero.
\end{proof}

\begin{lemma}\label{lemma:quotient_of_pointed_monoid_by_prime_ideal_is_the_corresponding_face}
    Let $P$ be a monoid, $I\subset P$ a prime ideal, and $F=P\backslash I$ the corresponding face. Then, the natural map of pointed monoids $f\colon F\cup\{\infty\} \to P\cup\{\infty\}/(I\cup\{\infty\})$ is an isomorphism. In particular, $\bb Z[F]=\bb Z[P]/I$.
\end{lemma}
\begin{proof}
Keeping in mind \ref{remark:faces_primes_of_monoid_are_faces_primes_of_pointification}, this follows immediately from \ref{definition:quotients_of_pointed_monoids_by_ideals}.  
\end{proof}

\begin{lemma}
    A pointed monoid $M$ lies in the essential image of $\cat{Mon} \to \cat{Mon}^0$ if and only if the ideal $\{\infty\}\subset M$ is prime.
\end{lemma}

\begin{proof}
    Both conditions are equivalent to the subset $M\backslash\{\infty\}$ being stable under the monoid law.
\end{proof}

\begin{corollary}\label{corollary:preimage_of_infty_is_prime_ideal}
    Let $M,N$ be monoids and $f \colon M\cup\{\infty\} \to N\cup\{\infty\}$ a morphism in $\cat{Mon}^0$. Then the ideal $f^{-1}\{\infty\}\subset M\cup\{\infty\}$ is prime.
\end{corollary}

\subsection{Faces of log schemes}\label{section: faces}

We come to the main construction of the paper.

\begin{definition}
    We say that a log scheme is \emph{Noetherian}, resp. \emph{locally Noetherian}, if its underlying scheme is.\footnote{All our log schemes have local charts by finitely generated monoids by assumption. In some of the literature, finite generatedness of the charts is not part of the definition of log schemes and must be included in the definition of Noetherianity.}
\end{definition}

\begin{definition}\label{definition:MXpointed}
    Let $X$ be a scheme. As a consequence of \ref{lemma:ring_to_mon0_to_mon_have_left_adjoints}, the forgetful functor from \'etale sheaves of pointed monoids to \'etale sheaves of monoids on $X$ has a left adjoint, which takes a sheaf $\ca F$ to the sheafification of the presheaf $\ca F \cup \{\infty\}$. From now on, we denote this sheafification by $\ca F^{\on{pointed}}$.
\end{definition}

\begin{remark}
    If $X$ is a \emph{connected} log scheme then $M_X^{\on{pointed}}(X) = M_X(X) \cup \{\infty\}$.
\end{remark}

\begin{definition}
    Let $X,Y$ be log schemes. A \emph{pointed morphism} $X \to Y$ is a pair consisting of a morphism of underlying schemes $\ul f\colon \ul X \to \ul Y$ and a morphism of sheaves of pointed monoids $\ul f^*M_Y^{\on{pointed}} \to M_X^{\on{pointed}}$ which makes the square
    \[
    \begin{tikzcd}
        \ul f^*M_Y^{\on{pointed}} \arrow[r] \arrow[d] & M_X^{\on{pointed}} \arrow[d] \\
        \ul f^*\ca O_Y \arrow[r] & \ca O_X
    \end{tikzcd}
    \]
    commute. We denote by $\LSchpointed$ the category of log schemes with pointed morphisms. The natural functor $\LSch \to \LSchpointed$ is faithful and essentially surjective (but not full) since the natural functor from sheaves of monoids to sheaves of pointed monoids on $\ul X$ is faithful.
\end{definition}

\begin{definition}\label{definition:face_union}
    Let $X$ be a log scheme and $f \colon F \to X$ be a pointed morphism. We say $f$ is an \emph{immersive face union} if for any strict geometric point $x \to X$ which lies in the scheme-theoretic image of $f$, there exist a strict \'etale neighbourhood $x \to T \to X$, a global chart $T \to \Spec \bb Z[Q]$ and a face $R \subset Q$ such that the pointed log scheme map $F \times_X T \to T$ is globally charted by a \emph{cartesian} square in $\cat{Sch}$
        \[
        \begin{tikzcd}
            F \times_X T \arrow[r] \arrow[d] & T \arrow[d] \\
            \Spec \bb Z[R] \arrow[r] & \Spec \bb Z[Q],
        \end{tikzcd}
        \]
    where the bottom horizontal arrow is given by \ref{lemma:quotient_of_pointed_monoid_by_prime_ideal_is_the_corresponding_face}. We say $f$ is a \emph{face} if it is an immersive face union and $F$ is connected.
\end{definition}

\begin{remark}\label{remark:being_a_face_union_is_a_local_ppty} The scheme map underlying an immersive face union is a closed immersion. The property of pointed morphisms ``being an immersive face union" is preserved by arbitrary base change, is local on the target in the strict \'etale topology, and is inherited by strict open and closed sub-log schemes of the source. In particular, if $X$ is locally Noetherian, then any immersive face union in $X$ is a disjoint union of faces.
\end{remark}
    
\begin{lemma}\label{lemma:face_unions_are_unions_of_faces}
    If $F \to X$ is an immersive face union and $F' \to F$ a strict open and closed immersion, then $F' \to X$ is an immersive face union. In particular, if $X$ is locally Noetherian, then every immersive face union in $X$ is a disjoint union of faces.
\end{lemma}

\begin{proof}
    The fact $F' \to X$ remains an immersive face union is clear from the definition. The rest follows since locally Noetherian schemes are the disjoint unions of their connected components.
\end{proof}

\begin{lemma}\label{lemma:faces_are_given_by_ideal_sheaves}
    Let $X$ be a log scheme. There is a natural contravariant equivalence between immersive face unions $f \colon F \to X$ and quasi-coherent sheaves of ideals $I \subsetneq \o M_X$ whose stalks at geometric points $x \to \ul X$ are either prime or all of $\o M_{X,x}$.
\end{lemma}

\begin{proof}
    Let $\ca I$ be the set of such sheaves of ideals $I \subsetneq \o M_X$. To each $I \in \ca I$, we may associate an immersive face union $f_I \colon F_I \to X$ as follows. The underlying scheme map $\ul F_I \to \ul X$ is the closed immersion cut out by $I\ca O_X$. For any geometric point $x \to \ul F_I$, pick a small enough strict \'etale neighbourhood $T \to X$ such that $T$ is globally charted by $\o M_{X,x}$. Since $I$ is quasi-coherent, after shrinking $T$ if necessary, we may assume $I(T) \subset \o M_{X,x}$ generates $I|_T$. In particular, $I(T)$ is the stalk of $I$ at $x$, so it is a prime ideal of $\o M_{X,x}$ and its complement $\o M_I(T)$ is a face. We enrich $\ul f_I|_T$ into a pointed log scheme map using the chart
        \[
        \begin{tikzcd}
            \ul F_I \times_{\ul X} \ul T \arrow[r,"f_I|_T"] \arrow[d] & \ul T \arrow[d] \\
            \Spec \bb Z[\o M_I(T)] \arrow[r] & \Spec \bb Z[\o M_{X,x}].
        \end{tikzcd}
        \]
    By descent, the resulting pointed log scheme map $f_I$ is well-defined and an immersive face union, and $f_I$ factors through $f_J$ if and only if $J \subset I$. All that is left to check is that an arbitrary immersive face union $f \colon F \to X$ is isomorphic to one of the $f_I$. This may be checked locally on $X$ in the strict \'etale topology, so by \ref{definition:face_union} we may assume the existence of a cartesian diagram of schemes
    \[
    \begin{tikzcd}
    	\ul F \arrow[r] \arrow[d] & \ul X \arrow[d] \\
    	\on{Spec} \bb Z[R] \arrow[r] & \on{Spec} \bb Z[Q]
    \end{tikzcd}
    \]
    where $R$ is a face of $Q$ and the vertical arrows are global charts for $F,X$. Then, the sheaf of ideals $I:=(Q \setminus R)M_X$ of $M_X$ is in $\ca I$ and $f$ is isomorphic to $f_I$.
\end{proof}

\begin{example}
Two extreme examples may help orient the reader. 

If $X$ is a point with log structure given by a monoid $P$, all the faces of $X$ are scheme-theoretically isomorphic to $X$, and the log structures they carry are given by the faces of $P$.

If $X=\bb A^n=\Spec \bb Z[e_1,...,e_n]$ with log structure given by the inclusion of the monomials into $\ca O_X$, then the prime ideals of $M_X(X)$ are of the form $(e_j)_{j\in J}$ where $J$ ranges through subsets of $\{1,...,n\}$. The face corresponding to $J$ is the scheme $\Spec \bb Z[e_i]_{i\notin J}$, with log structure given by its own monomials. Likewise, if $X$ is a regular scheme with log structure coming from a snc divisor $D$, the faces of $X$ are all the possible nonempty intersections of smooth components of $D$.
\end{example}

\begin{remark}
In the literature, it is common to call \emph{strata} of a log scheme $X$ the locally closed subschemes obtained by taking out (the underlying scheme of) all strict subfaces from (the underlying scheme of) a face. The characteristic monoid $\o M_X$ is locally constant on strata. A log scheme may have a lot more faces than it has strata, e.g. a log point only has one stratum.
\end{remark}

\begin{definition}
The {\it fine faces} of a coherent logarithmic scheme $X$ are the integralizations of the faces of $X$. 
\end{definition}

\begin{proposition}\label{proposition:integralization_of_faces_dominate_a_logsch}
    Let $X$ be a locally Noetherian log scheme, and $\ca F$ the set of faces of $X$. The closed immersions $\{\ul F^\int \to \ul X\}_{F\in \ca F}$ are jointly surjective.
\end{proposition}

\begin{proof}
    By \ref{remark:being_a_face_union_is_a_local_ppty}, we may replace $X$ by a strict \'etale covering and assume the existence of a global chart $f \colon X \to \Spec \bb Z[Q]$. Let $x\in X$ be a point, and let $\Spec \bb Z[R]$ be the smallest face of $\Spec \bb Z[Q]$ containing $f(x)$. By the minimality of $R$, we know $f(x)$ is contained in the open subscheme $\Spec \bb Z[R^\gp]$ of $\Spec \bb Z[R]$. A fortiori, $f(x)$ is in $\Spec \bb Z[R^\int]$. Therefore, $x$ is in the image of $\ul F^\int \to \ul F \to \ul X$, where $F$ is the connected component of $X\times_{\Spec \bb Z[Q]} \Spec \bb Z[R]$ containing $x$.
\end{proof}

Pointed morphisms are locally given by log scheme morphisms followed by a face, as we will show in \ref{prop:atomic_sources_factor_through_a_face}.

\begin{definition}
    A log scheme $X$ is called \emph{atomic} if it is globally charted by a sharp fs monoid $Q$, the vanishing locus $Z$ of $Q \setminus \{0\}$ is connected, and the closure of any point of $X$ meets $Z$. We call $Z$ (equipped with its $X$-strict log structure) the \emph{closed stratum} of $X$ and we say $X$ is an \emph{atomic neighbourhood of $x$} for any geometric point $x \to X$ which lands in $Z$.
\end{definition}
\begin{lemma}\label{lemma:logschemes_are_locally_atomic}
    Let $X$ be a locally Noetherian log scheme. Any strict geometric point in $X$ admits a strict \'etale atomic neighbourhood.
\end{lemma}
\begin{proof}
    The proof of \cite[Proposition 2.2.2.5]{Molcho_Wise_2022} carries over.
\end{proof}

\begin{proposition}\label{prop:atomic_sources_factor_through_a_face}
    Let $f \colon X \to Y$ be a morphism in $\LSchpointed$ between two atomic log schemes. There exists a face $F \to Y$ such that $f$ factors through a morphism $X \to F$ in $\LSch$. In particular, locally on the source and target in the strict \'etale topology, any pointed morphism between locally Noetherian log schemes factors as a non-pointed morphism followed by a face.
\end{proposition}
\begin{proof}
    The ``in particular" part follows from the rest by \ref{lemma:logschemes_are_locally_atomic}. Let $x \to X$ be a geometric point through the closed stratum, and pick a geometric point $y \to Y$ through the closed stratum which specializes $f(x)$. The monoid map $g \colon \o M_{Y,y} \to \o M_{X,x}\cup\{\infty\}$ is a global chart for $f$, and the ideal $I:=g^{-1}(\{\infty\})$ is prime by \ref{corollary:preimage_of_infty_is_prime_ideal}. The ideal sheaf $I\o M_Y \subset \o M_Y$ cuts out a face $F \to Y$ by \ref{lemma:faces_are_given_by_ideal_sheaves}, and there is a commutative diagram of solid arrows
	\begin{equation}\label{eqncdmon}
		\begin{tikzcd}
			X \arrow[r,dashed,"\psi"] \arrow[d] & F \arrow[r] \arrow[d] & Y \arrow[d] \\
			\on{Spec} \bb Z[\o M_{X,x}] \arrow[r] & \on{Spec} \bb Z[\o M_{Y,y} \setminus I] \arrow[r] & \on{Spec} \bb Z[\o M_{Y,y}]
		\end{tikzcd}
	\end{equation}
	where the vertical arrows are charts and the rightmost square is cartesian. The dashed arrow $\psi$ exists since $\bb Z[\o M_{Y,y} \setminus I] = \bb Z[\o M_{Y,y}]/(I)$ and $I$ vanishes in $\bb Z[\o M_{X,x}]$ by definition. Since the preimage of $\{\infty\}$ under $\o M_{Y,y} \cup \{\infty\} \to \o M_{X,x} \cup \{\infty\}$ is $I \cup \{\infty\}$, the bottom left horizontal arrow in \ref{eqncdmon} is given by a monoid map $\o M_{Y,y} \setminus I \to \o M_{X,x}$, i.e. $\psi$ is a morphism in $\LSch$ as claimed.
\end{proof}

\begin{example}
    In \ref{prop:atomic_sources_factor_through_a_face}, one cannot replace the atomic hypothesis on $X,Y$ by connectedness. For example, let $X$ be an elliptic curve equipped with the log structure with constant characteristic sheaf $\o M := \bb Nu \oplus \bb Nv$; $F_X \to X$ be the face cut out by the prime ideal $v\o M \subset \o M$, and $Y$ be the quotient of $X$ by a $2$-torsion point acting on $\o M$ via $u \mapsto v$ and $v \mapsto u$. The composite $F_X \to X \to Y$ does not factor through any nontrivial face of $Y$, and is not a morphism in $\cat{LSch}$.
\end{example}

\section{Examples from toric geometry and GW theory}\label{section: examples}

\subsection{Fibre products of toric varieties}\label{section:toric_fibreprod}

In this section, we specialize the discussion of faces of logarithmic schemes to the toric setting. More specifically, we examine a common and concrete source of coherent, non-fine logarithmic schemes -- fiber products of toric varieties in the category of schemes. In this case, everything in the previous section can be explained using fiber toric geometry, and specifically, fiber products of compactified fans. As we have alluded to before, the scheme theoretic fiber product of toric morphisms often lands outside the category of toric varieties. 

\subsubsection{Scheme-theoretic vs toric fibre product}\label{subsection:toric_fiber_products} Let $k$ be a field. A morphism of $k$-schemes $f \colon A \to B$ between two $k$-toric varieties is said to be a morphism of toric varieties if
\begin{itemize}
    \item $f$ restricts to a morphism $f^\circ$ of $k$-algebraic groups between the dense tori of $A$ and $B$.
    \item $f$ is $f^\circ$-equivariant with respect to the torus action on both sides.
\end{itemize}
The resulting category of normal toric $k$-varieties is well-known to be equivalent to the category of fans, i.e. complexes of strictly convex rational polyhedral cones, embedded in a latticed $\bb R$-vector space. It is slightly less well-known that the category of fans, or equivalently of toric varieties, admits fiber products, see~\cite{Msemistable} for a precise definition. We refer to it here as the {\it toric fiber product}. Fiber products of toric varieties in the categories of schemes and of toric varieties differ, as the following example illustrates.

\begin{example}
    Let $B\to\mathbb A^2$ be the blowup of $\mathbb A^2$ at the origin. Consider the cartesian square in $\cat{Sch}$:
    \[
    \begin{tikzcd}
        F\arrow{r}\arrow{d} & B \arrow{d}\\
        B\arrow{r} & \mathbb A^2.
    \end{tikzcd}
    \]
    The scheme $F$ certainly contains a copy of the dense torus $\mathbb G_m^2$. However, it also contains the square of the exceptional divisor, and so a copy of $\mathbb P^1\times\mathbb P^1$ that is disjoint from this $\mathbb G_m^2$. It follows that $F$ is not a toric variety. The toric fiber product is easy to compute: it is $B$ itself, as a special case of \ref{proposition:toric_fiberprod_is_fsification_of_sch_fiberprod}.
\end{example}

Before discussing fibre products of toric varieties, we must briefly discuss fibre products of tori themselves. The category of tori is not stable under scheme-theoretic fibre product, but the slightly larger category of \emph{diagonalizable group schemes} (defined below) is.

\begin{definition}\label{definition:diagonalizable}
    Let $S$ be a scheme. A $S$-group scheme $G$ is \emph{split diagonalizable} if it is of the form $\on{Hom}(M,\bb G_m)$ for some constant abelian $S$-group scheme $M$. The \emph{maximal torus} in $G/S$ is then the torus $\on{Hom}(M^{\on{tf}},\bb G_m)$ where $M^{\on{tf}}$ is the maximal torsion-free quotient of $M$. We say $G/S$ is \emph{diagonalizable} if it is split diagonalizable locally on $S$ in the \'etale topology. By descent, diagonalizable groups contain a maximal torus.
\end{definition}

\begin{definition}
    Given a scheme $S$ and a split diagonalizable group $G=\on{Hom}(M,\bb G_m)$ over $S$, the category of $S$-group homomorphisms $A \to G$ where $A$ is a torus has a final object $T \to G$. Namely, $T$ is the split torus $\on{Hom}(M^{\on{tf}},\bb G_m)$ where $M^{\on{tf}}$ is the quotient of $M$ by its torsion subgroup. By descent, if $G/S$ is diagonalizable but not necessarily split, this category still has a final object $T \to G$ which is a closed immersion. We call $T$ the \emph{maximal torus} inside $G/S$.
\end{definition}

\begin{lemma}\label{lemma:closed_subgroups_of_diagonalizable_groups_are_diagonalizable}
    Let $k$ be a field. Any closed subgroup or finite fibre product of diagonalizable $k$-group schemes remains diagonalizable.
\end{lemma}
\begin{proof}
    Any finite fibre product of diagonalizable $k$-group schemes is a closed sub-$k$-group of their product, so it suffices to check that an arbitrary closed subgroup $H$ of a diagonalizable group $G \to \on{Spec} k$ is diagonalizable. Replacing $k$ by a separable extension if necessary, we may assume $G$ is split diagonalizable, i.e. of the form $\on{Spec} k[M]$ for a finitely generated abelian group $M$, which we write multiplicatively. Let $J \subset R[M]$ be the ideal cutting out $H$ and $K \subset M$ be the subgroup consisting of those $k$ such that $k-1 \in J$. Then $H \hra G$ factors through the diagonalizable subgroup $G_0 := \on{Spec} k[M/K]$ of $G$, so we may replace $G$ by $G_0$ and assume $K=\{1\}$. Let $\phi$ be the map $k[M] \twoheadrightarrow k[M]/J =:A$. Since $\phi$ is a morphism of Hopf algebras over $k$, the $\phi(m)$ for $m \in M$ are group-like elements of $A$. Since $K=\{1\}$, they are pairwise distinct. Pairwise distinct group-like elements of a Hopf algebra over $k$ are $k$-linearly independent, so the basis $M$ of the $k$-vector space $k[M]$ maps to a linearly independent set in $A$ and $\phi$ is injective. Thus, $\phi$ is an isomorphism and $H=G$ is diagonalizable as claimed.
\end{proof}

\begin{proposition}\label{proposition:toric_fiberprod_is_fsification_of_sch_fiberprod}
    Let $X\to B$ and $Y\to B$ be morphisms of toric varieties, with underlying morphisms $U_X \to U_B$ and $U_Y \to U_B$ of dense tori. Let $F$ denote the scheme-theoretic fiber product
    \[
    F:= X\times_B Y.
    \]
    Let $T$ be the maximal torus inside $U_X \times_{U_B} U_Y$. The toric fiber product of $X \to B$ and $Y \to B$ is the normalization $F^{tor}$ of the closure $\o F$ of $T$ in $F$.
\end{proposition}

\begin{proof}
    By \ref{lemma:closed_subgroups_of_diagonalizable_groups_are_diagonalizable}, $U_X \times_{U_B} U_Y$ is diagonalizable so $T$ is well-defined. Consider a commutative diagram
    \[
    \begin{tikzcd}
        Z \arrow[r] \arrow[d] & X \arrow[d] \\
        Y \arrow[r] & B
    \end{tikzcd}
    \]
    of morphisms of toric varieties, and let $f \colon Z \to F$ be the resulting morphism of schemes. The dense torus of $Z$ lands in $U$, hence in $T$. Therefore, $f$ must factor uniquely through $\o F$. Since $Z$ is normal, $f$ further factors uniquely through $F^{\sf tor}$. The proposition follows, after observing that $F^{\sf tor}$ is a toric variety.
\end{proof}

The proposition tells us that the fiber product contains a ``toric component''. Eisenbud and Sturmfels go further by proving the following theorem, which we will generalize in \ref{theorem:faces_of_fiberprod_of_TVs_are_TVs_attached_to_asymptotic_cones}.

\begin{theorem}\label{theorem:normalised_orbits_are_TVs}
    Let $X\to B$ and $Y\to B$ be as above and let $F$ be the fiber product. Then the normalization of every irreducible component of $F$, equipped with its reduced structure, is also a toric variety. 
\end{theorem}

\begin{proof}
    This is \cite[Theorem 6.1]{EisenbudSturmfels}, together with the observation that normal, prime, binomial schemes over a field are toric.
\end{proof}


The different toric varieties occurring in the theorem above are saturations of fine faces of the fiber product $X\times_B Y$, where the toric varieties are given their standard logarithmic structures and the fiber product is taken as a coherent logarithmic scheme. 

\subsubsection{Fans and compactifications}\label{section:fans_and_compactifications}

Every toric variety\footnote{All our toric varieties are normal unless specified otherwise.} $X$ comes with a fan $\Sigma_X$ embedded in a vector space $N_\mathbb{R}$. As we have noted above, fiber products in the category of toric varieties can be calculated by the fiber product of fans. 

But we have seen above that the ordinary fiber product of a toric morphism is also combinatorial in nature. It raises a basic question: {\it how does one practically compute the scheme theoretic fiber product of toric morphisms?} 

It turns out that the appropriate combinatorics is the {\it extended tropicalization} or {\it compactified fan} of a toric variety. This is a notion due independently to Kajiwara and Payne and we recall it here. A detailed treatment can be found in \cite{Payne2009Analytification}.

Unless specified otherwise, in this section $k$ is a field and $X$ is an affine toric variety $X=\Spec k[P]$ for some torsion-free fs monoid $P$. We denote the dual monoid $\operatorname{Hom}(P,\bb N)$ by $P^\vee$ and the cocharacter lattice $\on{Hom}(P,\bb Z)$ of $X$ by $N$. The cone associated to $X$ is $\sigma_X:=\Hom(P, \bb R_{\geq 0}) = P^\vee \otimes_{\bb N} \bb R_{\geq 0}$.

\begin{definition}
    Denote by $\mathbf R_{\geq 0}$ the semiring $\bb R_{\geq 0} \cup \{\infty\}$, with the convention $\infty.0=0$. The \emph{extended cone} of $X$ is the pair $(N,\overline\sigma_X)$ where $\overline\sigma_X$ is the monoid $\Hom(P, \mathbf R_{\geq 0})$. We will often omit $N$ from the notation and call $\overline\sigma_X$ the extended cone of $X$.
\end{definition}

We write $\sigma, \o \sigma$ instead of $\sigma_X,\o \sigma_X$ when there is no ambiguity.

\begin{definition}\label{definition:origins_at_infinity}
    Let $\kappa$ be a face of $\sigma$. Denote by $\kappa^\circ$ the interior of $\kappa$, i.e. the complement of all the proper subfaces $\tau \subsetneq \kappa$. For all $x \in \kappa^\circ$, the point $\infty.x\in\overline\sigma$ corresponds to the map
    \begin{align*}
        P & \to \mathbf R_{\geq 0} \\
        f & \mapsto 0 \text{ if $f\in\kappa^\vee$} \\
        f & \mapsto \infty \text{ otherwise},
    \end{align*}
    so it only depends on $\kappa$. We call it the \emph{origin at $\kappa=\infty$} of $\overline\sigma$ and denote it by $\infty_\kappa$. Given another face $\kappa' \subset \sigma$, we denote by $\kappa \boxplus \kappa'$ the smallest face of $\sigma$ containing $\kappa+\kappa'$. Observe that $\infty_\kappa+\infty_{\kappa'}=\infty_{\kappa \boxplus \kappa'}$, since the sum of a point of $\kappa^\circ$ and one of $(\kappa')^\circ$ lands in $(\kappa \boxplus \kappa')^\circ$.
\end{definition}

We can express $\overline\sigma$ as the disjoint union of $\sigma$ and of some ``faces at infinity":

\begin{proposition}\label{proposition:extended_cones_are_unions_of_cones}
    We have
    \begin{equation}\label{eqn:decompose_basic_extended_cone_into_cones}
        \overline\sigma=\bigsqcup_{\kappa} \infty_\kappa + \sigma,
    \end{equation}
    where $\kappa$ ranges through the faces of $\sigma$. For each face $\kappa\subset\sigma$, the sub-semigroup $\infty_\kappa+\sigma$ of $\overline{\sigma}$ is a monoid with neutral element $\infty_\kappa$. There is an isomorphism of monoids $\infty_\kappa+\sigma=\sigma/(\bb R\kappa)$, which commutes with the natural surjective monoid maps from $\sigma$ on both sides.
\end{proposition}

\begin{proof}
    Let $p$ be a point of $\overline\sigma$, which we may see as a map $\sigma^\vee \to \mathbf R_{\geq 0}$. Since $\{+\infty\}$ is a prime ideal of $\mathbf R_{\geq 0}$, by \ref{corollary:preimage_of_infty_is_prime_ideal} we know $p^{-1}(\{+\infty\})$ is a prime ideal of $\sigma^\vee$, whose complement is a face $\kappa^\vee$. The face $\kappa\subset\sigma$ dual to $\kappa^\vee$ is the only face of $\sigma$ such that $p^{-1}(\{+\infty\})=\infty_\kappa^{-1}(\{+\infty\})$, i.e. such that $p$ is of the form $\infty_\kappa+q$ for some $q\in\sigma$. This proves that \ref{eqn:decompose_basic_extended_cone_into_cones} is a disjoint union. The canonical isomorphism $\infty_\kappa+\sigma=\sigma/(\bb R\kappa)$ follows from observing that $q$ is unique modulo $\bb R\kappa$.
\end{proof}

\begin{definition}
    With the notations of \ref{proposition:extended_cones_are_unions_of_cones}, we call $\infty_\kappa+\sigma=\sigma/(\bb R\kappa)$ the \emph{asymptotic cone of $\overline\sigma$ at $\kappa=\infty$}.
\end{definition}

All the semirings appearing in this paper are assumed to be commutative with unity.

\begin{definition}
    Let $M$ be a monoid and $\Lambda \to \Lambda'$ be a morphism of semirings. A \emph{$\Lambda$-external law on $M$} is a map
    \begin{equation}\label{external_law}
    \Lambda \times M \to M,
    \end{equation}
    written multiplicatively, such that for all $\lambda,\lambda' \in \Lambda$ and $m \in M$, $(\lambda.\lambda')m=\lambda(\Lambda'.m)$, $1.m=m$, and $f(-,m)$, $f(\lambda,-)$ are monoid maps (with respect to the additive structure on $\Lambda$). For example, all monoids, resp. cones, carry a canonical $\bb N$-external law, resp. $\bb R_{\geq 0}$-external law. Given an external law \ref{external_law}, the category of pairs consisting of a monoid map $M \to M'$ and an external law $\Lambda' \times M' \to M'$ extending \ref{external_law} has an initial object. We denote by $M\otimes_\Lambda \Lambda'$ the initial such $M'$.
\end{definition}

\begin{lemma}
    There are natural isomorphisms
    \[
    \o \sigma:=\on{Hom}(P,\mathbf R_{\geq 0}) \iso \sigma \otimes_{\bb R_{\geq 0}} \mathbf R_{\geq 0} \iso P^\vee \otimes_{\bb N} \mathbf R_{\geq 0}.
    \]
\end{lemma}
\begin{proof}
    Classically, $\sigma=P^\vee \otimes_{\bb N} \bb R_{\geq 0}$, so we only need to check that $\o \sigma$ is canonically isomorphic to $\sigma \otimes_{\bb R_{\geq 0}} \mathbf R_{\geq 0}$. There is a natural inclusion $\sigma \to \o \sigma$, and the $\mathbf R_{\geq 0}$-external law on $\o \sigma$ extends the $\bb R_{\geq 0}$-external law on $\sigma$, so it suffices to check the initialness of the pair $(\sigma \to \o \sigma, \mathbf R_{\geq 0} \times \o \sigma \to \o \sigma)$. Let $\phi \colon \sigma \to M$ be a monoid map and $\mathbf R_{\geq 0} \times M \to M$ an external law extending $\bb R_{\geq 0} \times \sigma \to \sigma$. We need to show that there exists a unique factorization
    \begin{equation}\label{facto}
        \sigma \to \o\sigma \xrightarrow{\o\phi} M
    \end{equation}
    compatible with the external laws. For all faces $\kappa \subset \sigma$, pick an element $x_\kappa \in \kappa^\circ$. By \ref{proposition:extended_cones_are_unions_of_cones}, $\o\sigma$ is generated as a monoid by $\sigma$ and by the elements $\infty_\kappa=\infty.x_\kappa$. Hence, if a factorization \ref{facto} exists, $\o\phi$ must be given by
    \[
    \o\phi(\infty_\kappa+x) := \infty.\phi(x_\kappa) + \phi(x)
    \]
    for all $x \in \sigma$ and all faces $\kappa \subset \sigma$. We must now check that
    \begin{enumerate}
        \item $\o\phi$ is well-defined as a set-theoretic map, i.e. $\infty.\phi(x_\kappa)+\phi(x) = \infty.\phi(y_\kappa)+\phi(y)$ if $y_\kappa \in \kappa^\circ$ and $x-y \in \bb R\kappa$.
        \item $\o\phi$ is a monoid map.
        \item $\o\phi$ commutes with the $\mathbf R_{\geq 0}$-external laws on both sides.
    \end{enumerate}
    For (1), we first show that $z \mapsto \infty.\phi(z)$ is constant on $\kappa^\circ$. Let $a_1,\dots,a_k$ be generators for the rays of $\kappa$. Any $z \in \kappa^\circ$ may be written as a linear combination of the $a_i$ with positive coefficients, so $\infty.\phi(z)=\sum\limits_{i=1}^k \infty.\phi(a_i)$, which does not depend on $z$. Pick $u,v \in \kappa^\circ$ such that $x+u=y+v$. We deduce     \begin{align*}
        \infty\phi(x_\kappa)+\phi(x) & = \infty\phi(u)+\phi(x) \\
        &= \infty\phi(u)+\phi(x+u) \\
        &= \infty\phi(v)+\phi(y+v) \\
        &= \infty\phi(y_\kappa)+\phi(y),
    \end{align*}
    which proves (1). (2) and (3) follow since $\o\phi$ is defined by extending $\phi$ to some linear combinations with coefficients in $\mathbf R_{\geq 0}$.
\end{proof}

\begin{example}\label{example:injections_of_cones_are_not_injections_of_extended_cones}
    The functor
    \begin{align}\label{eqn:compactification_of_cones_functor}
    \cat{Cones} & \to \cat{Monoids} \\
    \kappa & \mapsto \overline\kappa:=\kappa\otimes_{\bb R_{\geq 0}} \mathbf R_{\geq 0}
    \end{align}
    does not preserve injections, and does not commute with fiber products. For example, in $\bb Ru\oplus \bb Rv$ let $\sigma$ be the nonnegative orthant $\bb R_+u+\bb R_+v$ and let $\tau$ be the span of $u$ and $w:=v-u$. The inclusion $\sigma\subset\tau$ induces a non-injective map $\overline\sigma\to\overline\tau$, which sends all of $\infty v+\sigma$ to the point $\infty u+\infty w$. It follows that the fiber product (of monoids) $\overline\sigma\times_{\overline\tau}\overline\sigma$ does not lie in the essential image of \ref{eqn:compactification_of_cones_functor}.
\end{example}

Asymptotic cones have geometric significance: they describe the toric variety structure of closures of torus orbits. We recall this from the basic theory of toric varieties. 

\begin{proposition}\label{proposition:asymptotic_cones_give_the_TV_structure_of_orbits}
    Let $X$ be an affine toric variety over a field $k$, with cocharacter lattice $\Lambda$ and cone $\sigma\subset\Lambda_{\bb R}$. There is a natural bijection between torus orbits $Y^\circ \to X$ and faces $\kappa\subset\sigma$, such that the scheme-theoretic closure $Y \to X$ of $Y^\circ$ is an affine toric variety with cocharacter lattice $\Lambda/(\bb R\kappa\cap\Lambda)$ and cone $\sigma/(\bb R\kappa)=\infty_\kappa+\sigma$.
\end{proposition}

\begin{proof}
The proof is an elementary verification which we leave to the reader. 

%
\end{proof}

The next lemma states that the decomposition \ref{eqn:decompose_basic_extended_cone_into_cones} is well-behaved with respect to morphisms of cones.

\begin{lemma}\label{lemma:stratification_by_asymptotic_cones_is_respected}
    Let $f \colon \sigma \to \tau$ be a morphism of cones, and let $\overline f$ be the induced morphism $\overline\sigma \to \overline\tau$. Then, the preimage of an asymptotic cone under $\overline f$ is a disjoint union of asymptotic cones.
\end{lemma}

\begin{proof}
    By \ref{definition:origins_at_infinity}, we see that for any face $\kappa\subset\sigma$, $\overline f(\infty_\kappa)$ is $\infty_{\kappa'}$ where $\kappa'$ is the smallest face of $\tau$ containing $f(\kappa)$. Thus, the image of any asymptotic cone is contained in an asymptotic cone, and the lemma follows.
\end{proof}

\begin{definition}
    Let $(X_i\to B)_{1 \leq i \leq n}$ be morphisms of toric varieties over $\bb C$ and $F$ be their scheme-theoretic fibre product. We stratify $F$ by subschemes isomorphic to tori, which we call its \emph{torus strata}, as follows. First, we may partition $F$ into its subschemes $O := Y^\circ_1 \times_B \dots \times_B Y^\circ_n$ where $Y^\circ_i$ is a torus orbit in $X_i$. If $Y^\circ$ is nonempty, there must exist a torus orbit $Y^\circ_B$ of $B$ through which all $Y^\circ_i \to B$ factor, and $O = Y^\circ_1 \times_{Y^\circ_B} \dots \times_{Y^\circ_B} Y^\circ_n$. By \ref{lemma:closed_subgroups_of_diagonalizable_groups_are_diagonalizable}, we deduce that $O$ is a (necessarily trivial) torsor under a diagonalizable group scheme, with maximal torus $T$. Since $\on{Hom}(M,\bb C^\times) \simeq M$ for any finite abelian group $M$, each connected component of $O$ is isomorphic to $T$. We define the torus strata of $F$ to be these connected components, for all choices of $(Y^\circ_i)_{1 \leq i \leq n}$.
\end{definition}

We now get to the point: scheme-theoretic fiber products of toric varieties are governed by the combinatorics of their fans \emph{enriched to include the asymptotic cones}.

\begin{theorem}\label{theorem:faces_of_fiberprod_of_TVs_are_TVs_attached_to_asymptotic_cones}
    Let $(f_i \colon X_i \to B)_{1\leq i\leq n}$ be morphisms of affine toric varieties over $\bb C$. Let $F$ be the scheme-theoretic fibre product $X_1\times_B \cdots\times_B X_n$ and $Y^\circ \to F$ be a torus stratum. Let $Y_i^\circ \hra X_i$ and $Y^\circ_B \hra B$ be orbits such that $Y^\circ$ is a connected component of $Y^\circ_1 \times_{Y^\circ_B} \dots \times_{Y^\circ_B} Y^\circ_n$. Let $Y_B \to B$ be the closure of $Y_B^\circ \to B$ and $Y_i \to X_i$ that of $Y_i^\circ \to X_i$. The normalized closure $Y$ of $Y^\circ$ in $F$ is an affine toric variety, whose cone is the fiber product of the asymptotic cones of the $Y_i$ over that of $Y_B$.
\end{theorem}

\begin{proof}
    Asymptotic cones map to asymptotic cones by \ref{lemma:stratification_by_asymptotic_cones_is_respected}. By \ref{proposition:asymptotic_cones_give_the_TV_structure_of_orbits}, the toric variety structures of $Y_B,Y_i$ are given by asymptotic cones $\sigma_B,\sigma_i$ of the cones of $B,X_i$, such that $\sigma_i$ maps to $\sigma_B$. Thus, by \ref{proposition:toric_fiberprod_is_fsification_of_sch_fiberprod}, $Y$ is the affine toric variety with cone $\sigma_1\times_{\sigma_B}...\times_{\sigma_B}\sigma_n$.
\end{proof}

\begin{corollary}
    Let $k$ be a field and let $(f_i \colon X_i \to B)_{1\leq i\leq n}$ be morphisms of toric varieties over $k$, with scheme-theoretic fiber product $f \colon F \to B$. Let $Y \to F$ be the normalization of an irreducible component, and $Y_B \to B$ be the smallest closure of torus orbits containing $f(Y)$. Let $\sigma_B$ be the asymptotic cone of $B$ attached to $Y_B$. The asymptotic cones of $X_i$ mapping to $\sigma_B$ form a fan $\Sigma_i$, and $Y$ is a toric variety with fan $\Sigma_1\times_{\sigma_B}\cdots\times_{\sigma_B}\Sigma_n$.
\end{corollary}

The different (compactified) fans occurring in the corollary above are the (compactified) fans of the toric varieties occurring as the faces of the fiber product $X_1\times_B\cdots\times_B X_n$, where the latter is given its logarithmic structure as a fiber product in the category of coherent logarithmic schemes. 

\subsubsection{Chow quotients of toric varieties}\label{sec: Chow-quotients} Another instance of fine and saturated limits occurs in the theory of Chow quotients. In~\cite{KSZ91} the authors consider the problem of forming the quotient of a projective $T$-toric variety $X$ by a subtorus $H\subset T$ of the dense torus. 

The Chow quotient $X/\!/H$ is constructed as the normalized closure of the natural map $T/H\hookrightarrow \mathsf{Chow}(X)$ into the Chow variety of cycles in $X$, where the map is obtained by considering $H$-orbit closures.

After choosing linearization data $\theta$, geometric invariant theory gives rise to a quotient $X/\!/_\theta H$. Kapranov--Sturmfels--Zelevinsky show that for each $\theta$, there is an equivariant map
\[
X/\!/H\to X/\!/_\theta H.
\]
The system of GIT quotients forms an inverse system with equivariant maps. While this produces a map
\begin{equation}\label{eqn:chowgit}
    X/\!/H\to \varprojlim_\theta X/\!/_\theta H,
\end{equation}
the Chow quotient is {\it not} typically the inverse limit itself. 

\begin{proposition}
    The Chow quotient $X/\!/H$ is isomorphic to the fine and saturated inverse limit
    \[
    \varprojlim_\theta X/\!/_\theta H.
    \]
\end{proposition}

\begin{proof}
    The Chow quotient is a normal $T/H$-toric variety. Kapranov--Sturmfels--Zelevinsky have already shown that it admits a finite map to the inverse limit of GIT quotients~\cite[Corollary~4.3]{KSZ91}. The fine and saturated inverse limit is also the normalized closure of $T/H$ in the scheme theoretic inverse limit. This normalized closure is $X/\!/H$ since \ref{eqn:chowgit} identifies the copies of $T/H$ on both sides.
\end{proof}

\subsection{Stable maps} Our motivation to pursue a description of these ``other components'' comes from logarithmic Gromov--Witten theory. We now explain these examples. We work over $\CC$ for this discussion.

\subsubsection{The setup and context} Let $X$ be a smooth projective variety over the complex numbers and $D$ a simple normal crossings divisor on $X$ with components $D_1,...,D_r$.\footnote{The smoothness and normal crossings hypotheses are here for simplicity of exposition. The story is essentially unchanged for logarithmically smooth $X$. In fact, one can even drop this assumption, but the virtual class, which is a key part of the story, is no longer well-defined.} We are interested in maps of pairs from pointed nodal curves:
\[
f\colon (C,p_1,\ldots,p_n)\to (X,D).
\]
Maps of pairs have the property that $f^{-1}(D)$ is a subset of $\{p_i\}$, and we additionally fix the contact orders $c_{ij}$ of $f$ at the point $p_i$ with the divisor $D_j$. The moduli space is not compact, and a compactification is provided by examining families of {\it logarithmic maps} over fine and saturated logarithmic schemes $S$. When $S$ has trivial logarithmic structure we recover families of maps from pairs. 

The moduli problem of logarithmic maps from logarithmic curves to a logarithmic target is a stack over the category of fine and saturated logarithmic schemes. It is represented by an algebraic stack with logarithmic structure by the main results of~\cite{AC11,Che10,GS13}. 

We impose a {\it stability} condition, that the underlying map has finite $X$-automorphism group, and the outcome is a Deligne--Mumford stack with two key properties:
\begin{itemize}
\item The components of the moduli space parameterizing maps with fixed discrete data, recorded formally below, are proper. 
\item The connected components of the moduli space admit virtual fundamental classes. 
\end{itemize}

The special case when $D$ is a single smooth divisor ($r=1$) is simpler and much better understood, and in fact essentially completely understood in principle~\cite{MP06}. The simple normal crossings case is significantly more complicated, and this can be seen for instance in the relative complexities of the degeneration formulas in these settings~\cite{AbramovichChenGrossSiebert,Li01,MR23}. It suggests a basic question: {\it can the calculations in the smooth pair case be bootstrapped to produce calculations in the general case?} This general strategy has been the topic of several papers~\cite{BattistellaNabijouRanganathan,NabijouRanganathan}. 

From our perspective, relating the general case to the case $r=1$ involves performing a scheme theoretic fiber product of the smooth pair spaces, and then applying the fine saturation functor (\ref{proposition:stable_maps_to_XD_are_fiber_prods_of_maps_to_XDi}). The main complexity in controlling the full logarithmic theory is the latter operation. Our construction proposes, for each snc pair $(X,D)$ several ``spurious components'' of non-logarithmic maps that capture the relevant complexity, and explain why the naive generalizations of statements like the degeneration formulas fail.

\subsubsection{The details} We denote by $M_i$ the log structure on $X$ induced by $D_i$, and by $M$ the one induced by $D$. We denote the resulting log scheme by $(X,D)$ or $(X,M)$ interchangeably. 

We fix a {\it numerical data} $\Lambda$ of logarithmic stable maps to $X$, i.e. the data of (i) the domain genus $g$, (ii) the number $n$ of marked points, (iii) the class $\beta\in H_2(X;\ZZ)$ of $f_\star[C]$, and (iv) the tangency order $c_{ij}$ for $p_i$ along $D_j$. 


\begin{definition}\label{definition:log_stable_maps}
    The \emph{moduli space $\ca M_\Lambda(X,D)$ of stable maps with numerical data $\Lambda$} (we will often omit $\Lambda$ from the notation) is the \'etale DM stack over $\LSch$ whose fiber over $T$ is the groupoid of stable maps
    \begin{equation}\label{eqn:stable_maps}
        \begin{tikzcd}
            C \arrow[d] \arrow[r] & (X,M) \\
            T &
        \end{tikzcd}
    \end{equation}
    with numerical data $\Lambda$. The morphisms are commutative diagrams which are cartesian with respect to the top arrow of \ref{eqn:stable_maps} and the identity on $(X,M)$.
\end{definition}

The fact that this is representable by a Deligne--Mumford stack over ordinary schemes is the main result of~\cite{AC11,GS13}. 

The geometry of $\ca M(X,D)$ relates to that of the $\ca M(X,D_i)$ via the following proposition.

\begin{proposition}\label{proposition:stable_maps_to_XD_are_fiber_prods_of_maps_to_XDi}
    The natural map
    \[
    \ca M(X,D) \to \ca M(X,D_1) \times_{\ca M(X)} \ca M(X,D_2) \dots \times_{\ca M(X)} \ca M(X,D_n)
    \]
    is an isomorphism, where the fiber product is taken in $\cat{LSch}^{\sf fs}$.
\end{proposition}

\begin{proof}
See~\cite{AC11}. 
\end{proof}

In view of \ref{proposition:stable_maps_to_XD_are_fiber_prods_of_maps_to_XDi}, once we understand the $\ca M(X,D_i)$ there is still one ingredient missing to understand $\ca M(X,D)$: we must be able to compute fiber products in $\cat{LSch}^{\sf fs}$, i.e. to compute the integral saturation of the scheme-theoretic fiber product (see \ref{remark:fiber_prod_vs_fs_fiber_prod}). From now on, we denote by $\ca M^{\sf nfs}(X,D)$ the fiber product of the $\ca M(X,D_i)$ over $\ca M(X)$ in $\cat{LSch}$, so $$\ca M(X,D)=(\ca M^{\sf nfs}(X,D))^{\sf fs}.$$ 

In order to effectivize this, we need to chart the morphism
\[
\mathcal M(X,D_i)\to\mathcal M(X). 
\]
Let $p$ be a point of the domain and let $q$ be its image. Let $P$ and $Q$ be the stalks of the characteristic sheaves of $\mathcal M(X,D_i)$ at $p$ and of $\mathcal M(X)$ at $q$. These have explicit descriptions: the dual monoid $Q^\vee=\on{Hom}(Q,\bb N)$ can be interpreted as a moduli space of tropical curves whose numerical data is that of the domain curve of the map at $q$. Similarly, $P^\vee$ is a moduli space of tropical maps. See~\cite[Section~2]{GS13} for further discussion. Locally near $p$ and $q$, $\mathcal M(X,D_i)\to\mathcal M(X)$ is charted by the toric map
\[
\Spec \bb Z[P] \to \Spec \bb Z[Q],
\]
corresponding to the cone map $P^\vee \to Q^\vee$. Therefore, the cone complexes of the fine faces of $\ca M^{\sf nfs}(X,D)$ can be obtained by describing tropical moduli spaces in the smooth divisor case, passing to extended tropicalizations, and then taking fiber products. This immensely aids the analysis. 

We single out a key case in this strategy: the double ramification cycle. Let $A$ be an $n\times r$ integer matrix, and fix an integer $g$. The locus $\mathsf{DR}_g(A)$ in $\mathcal M_{g,n}$ is the locus of curves $[C,p_1,\ldots,p_n]$ such that the for all $j$ the divisor $\sum_i a_{ij}p_i$ is principal. It can be compactified by studying logarithmic line bundles~\cite{Molcho_Wise_2022}, or more explicitly via a resolution of the Abel--Jacobi map~\cite{Hol21,MW20}. 

One can study the fine saturation map:
\[
\mathsf{DR}_g(A)\to \mathsf{DR}_g(A_1)\times_{\overline{\mathcal{M}}_{g,n}}\cdots\times_{\overline{\mathcal{M}}_{g,n}} \mathsf{DR}_g(A_r).
\]
This fiber product relates the {\it higher} double ramification cycle to {\it products} of several double ramification cycles, and aspects of this geometry are analyzed in~\cite{MR21}, however a complete understanding of the right hand side is not required. It would be interesting to analyze this more completely. We state the problem formally:\medskip

\noindent
{\bf Problem.} {\it Describe the fine faces of the logarithmic scheme $\mathsf{DR}_g(A_1)\times_{\overline{\mathcal{M}}_{g,n}}\cdots\times_{\overline{\mathcal{M}}_{g,n}} \mathsf{DR}_g(A_r)$.}\medskip

In order to do so, one needs to understand the corresponding fiber product of extended tropicalizations. The log structure on each $\mathsf{DR}_g(A_1)$ can be described from the corresponding tropical moduli problem: the inclusion of the tropical double ramification cycle $\mathsf{DR}^{\sf trop}_g(A_i)$ into $\mathcal M_{g,n}^{\sf trop}$ is a closed (in the Euclidean topology) sub-cone complex of a subdivision, so it gives a map on Artin fans that pulls back to the inclusion of a Zariski open subset of a logarithmic blowup:
\[
\overline{\mathcal{M}}_{g,n}(A_i)\to\overline{\mathcal{M}}_{g,n}.
\]
The map $\mathsf{DR}_g(A_i) \to \overline{\mathcal{M}}_{g,n}$ factors through a \emph{strict} map $\mathsf{DR}_g(A_i)\to \overline{\mathcal{M}}_{g,n}(A_i)$, so that the log structure of $\mathsf{DR}_g(A_i)$ can be understood via the tropical problem $\mathsf{DR}^{\sf trop}_g(A_i)$. In \ref{section:worked_examples}, we will give detailed examples of computations of tropical double ramification cycles $\mathsf{DR}^{\sf trop}_g(A_i)$ which illustrate the need for a fine saturation algorithm. 

\subsubsection{Sources of non-fine morphisms} We conclude the section by describing the two ways in which a point of $\ca M^{\sf nfs}(X,D)$ can {\it fail} to lie in $\ca M(X,D)$. We assume familiarity with the basics of tropical maps, see~\cite{GS13} or~\cite{RanganathanExpansions} for relevant background. 

\begin{remark}\label{remark:tropical_obstruction}
We see a potential \emph{tropical obstruction} for a point $T \to \ca M^{\sf nfs}(X,D)$ to factor through $\ca M(X,D)$. Indeed, $T$ corresponds to a stable map $\ul\pi$ from a nodal curve $\ul C/\ul T$ to $\ul X$, together with a $r$-uple of log structures $M_C^{(1)},\dots,M_C^{(r)}$ and log stable maps $\pi_i \colon (\ul C,M_C^{(i)}) \to (X,D_i)$ with underlying scheme map $\ul \pi$. A factorization of $T \to \ca M^{\sf nfs}(X,D)$ through $\ca M(X,D)$ would be a log structure $M_C$ on $C$ and a stable map $\pi \colon (\ul C,M_C) \to (X,D)$ inducing all the $\pi_i$. The maps $\pi$, resp. $\pi_i$ would tropicalize to \emph{piecewise-linear functions} $\pi^{\sf trop}$, resp. $\pi_i^{\sf trop}$ on the dual graph $\Gamma_{C/T}$ of $C/T$, with values in $\bb N^r$, resp. $\bb N$.\footnote{This means functions on the metric graph $\Gamma_{C/T}$ which are linear with integer slope along every edge. For a detailed description of this tropicalization procedure, see e.g. \cite{Molcho_Wise_2022}.} These functions record the valuations of local parameters for the $D_i$ on components of $C$. The piecewise-linear maps $\pi^{\sf trop}$ and $\pi_i^{\sf trop}$ would have the same slopes along every edge since the latter would be the projection of the former onto the $i$-th coordinate of $\bb Z^r$. This gives a combinatorial obstruction for the existence of $\pi$: no suitable $\pi^{\sf trop}$ can exist unless all the $\pi_i^{\sf trop}$ have the same slopes along every edge of $\Gamma_{C/T}$.
\end{remark}

The tropical obstruction of \ref{remark:tropical_obstruction} is not the only obstacle to the surjectivity of the map $\ca M^{\sf nfs}(X,D) \to \ca M(X,D)$, as the next example will illustrate.

We use the underline notation for the underlying schemes and underlying scheme maps of objects and morphisms in $\cat{LSch}$. When convenient, we will also use underlined letters for some schemes which do not arise in this way, to emphasize that they do not carry a log structure.

\begin{example}
     Let $X=\bb P^2$ with affine coordinates $x,y$ and let $D=D_x+D_y$ be the divisor of the two coordinate lines. The cone complex $\Sigma_X$ is $\bb R_{\geq 0}^2$. We fix numerical data $\Lambda$ as follows: the genus is $0$, the number of marked points $n$ is $1$, the curve class is $2[H]$, and the tangency order of the marked point with each of the divisors is $2$. 
     
    We now describe a map that is not fine logarithmic. The example is based on one from~\cite{NabijouRanganathan}. In $\mathbb P^1\times\mathbb P^1$ with affine coordinates $s,t$, let $C_0,C_1,C_2$ be the lines $t=0,s=0$ and $t=\infty$, with generic points $\eta_0,\eta_1,\eta_2$. Let $\ul C$ be their union. It is a nodal curve with two nodes $q_i=C_{i-1}\cap C_i$, $i\in \{1,2\}$. We mark $\ul C$ along the section $t=1$.

    Consider a strict map $(\Spec \bb C,M) \to \ca M(X,D)$, corresponding to a map $\phi \colon C \to X$, such that
    \begin{enumerate}
        \item $C$ has underlying scheme $\ul C$;
        \item $C_0,C_2$ do not land inside $D$;
        \item $C_1$ is contracted to the origin;
        \item $\phi^{\sf trop}$ has slope $1$ along both edges (oriented toward $C_1$).
    \end{enumerate}
    Let $\lambda,\mu\in\bb C^\times$ be the slopes of the lines $\phi(C_0)$ and $\phi(C_2)$ of $X=\bb P^2$. We make two claims: first, that $\lambda=\mu$, i.e. that $\phi(C_0)$ and $\phi(C_2)$ are the same line, and second, that this constraint disappears when we replace $D$ by $D_x$ (or $D_y$): for any $\lambda,\mu\in\bb C^\times$ there exists a log stable map $\phi_x \colon C \to (\ul X,D_x)$ satisfying (1) to (4), sending $C_0$ to the line $y=\lambda x$ and $C_2$ to the line $y=\mu x$. Suppose momentarily that both claims hold. Then the stratum of maps in 
    $$
    \ca M^{\sf nfs}(X,D)=\ca M(X,D_x) \times_{\ca M(X)} \ca M(X,D_y)
    $$ 
    satisfying (1) to (4) is two-dimensional, parametrized by $\lambda,\mu$ by the second claim, and its restriction to $\ca M(X,D)$ is the line $\lambda=\mu$ by the first claim, despite the absence of any tropical obstruction as in \ref{remark:tropical_obstruction}.
    
    Consider the first claim. Denote by $\pi$ the morphism $C \to (\Spec \bb C,M)$, and by $U$ the strict dense open $U:=\bb A^2_{x,y} \subset X$. The log scheme $U$ is globally charted by the free monoid with two generators $\on{log} x$ and $\on{log} y$, with images $x$ and $y$ in $\ca O_X(U)$. The restriction $V:=C\times_X U$ is the complement of one smooth point of $C_0$ and one smooth point of $C_2$, so $C$ is covered by its Zariski opens $V$, $V_0:=C_0\backslash \{q_1\}$ and $V_2:=C_2\backslash \{q_2\}$. Let $f_V$ be the image of $\on{log}y - \on{log}x$ under $\phi^* \colon M_X^{gp}(U) \to M_C^{gp}(V)$. By definition of $\lambda$, the restriction of $f_V$ to $V \cap V_0$ equals $\lambda$, seen as a section of $M_C^{gp}$ via
    \[
    \bb C^\times = \ca O_C^\times(C) = M_C^\times(C) \hra M_C^{gp}(V \cap V_0) 
    \]
    Likewise, $f_V$ restricts to $\mu$ on $V \cap V_2$, so we may glue $f_V$ with the sections $\lambda\in M_C^{gp}(V_0)$ and $\mu\in M_C^{gp}(V_2)$ to obtain a global section $f$ of $M_C^{gp}$. The map $\pi^*M_S^{gp} \to M_C^{gp}$ is surjective on global sections, e.g. by \cite[Lemma 4.6.1]{Molcho_Wise_2022}, so $f$ is constant with value $\lambda=\mu$.

    For the second claim, let $\ul C^\circ:=\ul C\backslash\{s=\infty\}$, so that $\ul C^\circ$ is the vanishing locus of $stt'$ inside $\bb A^1_s \times \bb P^1_{(t:t')}$. The map $\ul\phi^\circ \colon \ul C^\circ \to \bb A^2_{x,y}$ defined by $x \mapsto s, y \mapsto (t+\lambda s t')(t'+\mu s t)$ extends to a map $\ul\phi \colon \ul C \to \ul X$ which satisfies (1) to (3), and sends $C_0,C_2$ to the lines $y=\lambda x$ and $y=\mu x$ respectively. Make $\ul C \to \Spec \bb C$ a vertical log smooth curve $C \to (\Spec \bb C,\bb N\alpha)$ by giving length $\alpha$ to both nodes. We may now enrich $\ul\phi$ into a log scheme map $\phi \colon C \to (\ul X,D_x)$ satisfying (4) by sending $\on{log}x$ to $\on{log}s$, so the claim holds.
\end{example}

\section{Algorithms on monoids}\label{section:algorithms_on_monoids}

We now explain how to explicitly calculate the fine faces of a logarithmic scheme. As we have noted, in the examples we have in mind for applications -- moduli of varieties and stable maps -- the log schemes of interest arise as non-fine fibre products of fine log schemes. Computing charts and presentations of the charting monoids of such a fiber product from charts and presentations of the maps in the fibre product is straightforward. Hence, the relevant input to our algorithms amounts to local charts for a logarithmic scheme, and knowledge of generators and relations for the charting monoids.

From now on, all our monoids are assumed to be finitely generated unless specified otherwise.

\subsection{Pure difference schemes}
\begin{definition}
    A \emph{pure difference binomial} with coefficients in any ring is the difference of two distinct monic monomials.\footnote{We emphasize that monic monomials are not pure difference binomials.} Let $n\geq 0$ be an integer. A \emph{pure difference ideal} in $\bb A^n$ is an ideal generated by pure difference binomials. A \emph{pure difference scheme} is a closed subscheme of $\bb A^n$ cut out by a pure difference ideal.
\end{definition}

\begin{definition}\label{definition:monoidal_relations}
    Let $F$ be a monoid, and $R\subset F\times F$ a relation. We say $R$ is a \emph{monoidal relation} if for all $f\in F$, we have $(f,f)+R\subset R$. We observe that when $R$ is an equivalence relation, it is monoidal if and only if the monoid law on $F$ is compatible with the set-theoretic quotient $F/R$.
\end{definition}

\begin{example}
    An order $\preceq$ on a monoid $F$ is \emph{monoidal} if it is also a monoidal relation, i.e. if for each $a,b,c\in F$, $a \preceq b \implies a+c\preceq b+c$. The lexicographic order on $\bb N^n$ is monoidal. Monoidal orders are also sometimes called \emph{monomial orders}. We will use both names interchangeably.
\end{example}

\begin{remark}\label{remark:surj_monoid_maps_are_pure_diff_ideals}
    There is a natural equivalence between pure difference ideals of $\bb A^n$ (seen as a category whose arrows are inclusions) and surjective monoid maps $\NN^n \twoheadrightarrow M$. Dually, there is a natural equivalence between pure difference subschemes of $\bb A^n$ and monoidal relations on $\NN^n$.
\end{remark}

\begin{remark}
Pure difference subschemes of $\bb A^n$ can be endowed with the pullback of the natural log structure on $\bb A^n$, making them log schemes. By \ref{remark:surj_monoid_maps_are_pure_diff_ideals}, those are exactly the log schemes of the form $\Spec \ZZ[M]$ for some surjective monoid map $\NN^n \to M$.
\end{remark}

Throughout the remainder of this section, unless specified otherwise, we will be in the following setup.

\begin{situation}\label{situation:presentation_of_monoid}
    Let $\Phi\colon F \to M$ be a surjective map of monoids with $F$ free of finite rank $n$, and $R\subset F\times F$ the equivalence relation defining $M$. Let $\preceq$ be a total monoidal order on $F$. Denote by $(F\times F)^\succ$ the subset of $F\times F$ consisting of those $(a,b)$ with $a\succ b$. Observe that any equivalence relation on $F$ is the reflexive, symmetric closure of its intersection with $(F\times F)^\succ$. Denote by $\Delta$ the diagonal of $F\times F$.
\end{situation}

 In \ref{situation:presentation_of_monoid}, we think of $F$ as the monic monomials in a polynomial ring in $n$ variables, and of a pair $(a,b) \in F \times F$ as the difference $a-b$. If $a \neq b$, we call $a$ and $b$ the \emph{terms} of $(a,b)$.

Some important algorithms related to computing and manipulating ideals (most notably multivariate division with remainder and Buchberger's algorithm for the computation of Gr\"obner bases) do not really rely on addition, but rather on \emph{term elimination}: given two polynomials $f$ and $g$, multiply them by monomials until their ``biggest" terms match, and the difference of the resulting polynomials will be an element of the ideal $(f,g)$ which is potentially ``smaller". In \cite{EisenbudSturmfels}, Eisenbud and Sturmfels use the fact that term elimination preserves binomials to deduce closure properties of the class of binomial ideals in the affine space over a field. They give term-elimination-based algorithms to compute e.g. the radical and primary decomposition of a binomial ideal.

We may translate this binomial term elimination procedure in the language of monoids. Let $(a,b)$ and $(c,d)$ be in $R$. Suppose that $b\preceq a$ and $d\preceq c$. Since $R$ is monoidal, $(a\vee c,b+(a\vee c)-a)$ and $(a\vee c, d+(a\vee c)-c)$ are also in $R$. Therefore, by transitivity, $R$ contains the ``term elimination relation"
\begin{equation}\label{term_elim_rel}
    (b+(a\vee c)-a, d+(a\vee c)-c),
\end{equation}
whose largest term has a chance of being smaller than $a$ and $c$ with respect to $\preceq$. This suggests that monoidal analogues of Gr\"obner basis algorithms would be a useful computing tool in tropical and logarithmic geometry.

We will define several objects which depend on the order $\preceq$, but we will omit it from the notation when we can unambiguously do so.

\subsection{Gr\"obner bases}

\begin{definition}
    In \ref{situation:presentation_of_monoid}, we call \emph{rewriting rule}, or just \emph{rule}, an element of $(F \times F)^\succ$. When we call an element $(c,d) \in (F\times F)^\prec$ a rule, we are referring to the rule $(d,c)$.
    
    Consider any $a,b$ in $F$ and any set of rules $B$. We say that $a$ \emph{may be rewritten as $b$ by applying a rule in $B$}, and we write $a\to_B b$, whenever $(a,b)\in B+\Delta$\footnote{in other words, $\to_B$ is the monoidal closure of the relation $B\subset F\times F$}. We denote the symmetric closure of $\to_B$ by $\leftrightarrow_B$ and its reflexive transitive closure by $\to_B^*$. We refer to $\to_B^*$ as \emph{applying rules of $B$ repeatedly}. 
\end{definition}

\begin{definition}
	A \emph{basis} for $M$ is a subset of $F\times F$ whose symmetric, transitive, monoidal closure is $R$. For any $(a,b)$ in $F\times F$ with $a\neq b$, we call $\ini(a,b):=\max(a,b)$ the \emph{initial term} of $(a,b)$, where the maximum is taken with respect to $\preceq$. If $a=b$, we declare $\ini(a,b)$ to be undefined.
    
    Since $R$ is monoidal, the set $\ini(R):=\{\ini(r)|r\in R\}$ is an ideal of $F$. A \emph{Gr\"obner basis} for $M$ (with respect to $\preceq$) is a finite subset $\{r_i\}_{1\leq i\leq k}$ of $R\cap (F\times F)^\succ$ such that the ideal $\ini(R)$ is spanned by the $\ini(r_i)$.
\end{definition}

\begin{lemma}[Macaulay]\label{lemma:normal_forms_exist}
	The restriction of $\Phi \colon F \to M$ to $F\backslash\ini(R)$ is bijective.
\end{lemma}

\begin{proof}
    Consider any $m\in M$. Then $\Phi^{-1}(\{m\})$ has a minimum $a$. For any $b\in\Phi^{-1}(\{m\})$, we have $(a,b)\in R$ so $\ini(a,b)\in\ini(R)$ if and only if $a\neq b$. Thus, the intersection of $\Phi^{-1}(\{m\})$ with $F\backslash\ini(R)$ is precisely $\{a\}$ and $\Phi|_{F\backslash\ini(R)}$ is bijective.
\end{proof}

\begin{definition}\label{definition:normal_form}
    Let $a$ be an element of $F$. The \emph{$R$-normal form} of $a$, or just \emph{normal form} if $R$ is clear from context, is defined as the minimum of the $R$-equivalence class of $a$. The map $F \to M \to F\backslash\ini(R)$, where the second arrow is the inverse of the bijection in \ref{lemma:normal_forms_exist}, takes an element to its normal form.
\end{definition}

\begin{lemma}\label{lemma:Grobner_bases_are_bases}
    A Gr\"obner basis for $M$ is a basis.
\end{lemma}

\begin{proof}
    Let $B$ be a Gr\"obner basis. We want to show that $B$ generates $R$, which reduces to showing that for any $a\in F$, we have $a \to_B^* b$ where $b$ is the $R$-normal form of $a$. This follows from \ref{lemma:normal_forms_exist} and the fact that $\ini(R)=\ini(B)+F$ since $B$ is a Gr\"obner basis.
\end{proof}

\begin{lemma}\label{lemma:ACC_for_free_monoids}
    Any ascending chain of ideals of $\NN^n$ is stationary.
\end{lemma}

\begin{proof}
    There is an order-preserving bijection between ideals of the monoid $\NN^n$ and ideals of the ring $\ZZ[x_1,...,x_n]$ which are generated by monic monomials, so this follows from the Noetherianity of $\ZZ[x_1,...,x_n]$.
\end{proof}

\begin{definition}
    A Gr\"obner basis $B$ of $M$ is called \emph{reduced} if for any $r\in B$, the ideal $\ini(r)+F$ of $F$ contains none of the terms of $B\backslash\{r\}$.
\end{definition}

\begin{proposition}
    In \ref{situation:presentation_of_monoid}, denote by $\{a_i\}_{i\in I}$ the minimal generating set of the ideal $\ini(R)\subset F$. For each $i$, call $b_i$ the $R$-normal form of $a_i$. Then, $B:=\{(a_i,b_i)\}_{i\in I}$ is the unique reduced Gr\"obner basis for $M$ with respect to $\preceq$.
\end{proposition}

\begin{proof}
    Any ideal $J \subset F$ admits a unique minimal generating set, namely $\{j \in J:j \notin (J \setminus \{j\})+F\}$, and this set is finite by \ref{lemma:ACC_for_free_monoids}. If $B'$ is a reduced Gr\"obner basis, then the initial terms of $B'$ must contain the $a_i$ (by Gr\"obner-ness), and they must be exactly the $a_i$, each of which can appear only once (by reducedness). Thus, we have reduced to showing that a set $B''\subset (F\times F)^\succ$ of the form $\{(a_i,c_i)\}_{i\in I}$ is a reduced Gr\"obner basis if and only if $B''=B$. But $B''$ is a reduced Gr\"obner basis if and only if for each $i\neq j$, the term $c_i$ is not in $a_j+F$. This is equivalent to the $c_i$ not being in $\ini(R)$, which, by \ref{lemma:normal_forms_exist}, is in turn equivalent to $c_i=b_i$ for each $i$.
\end{proof}

\subsection{Buchberger's algorithm}

\begin{algorithm}[Buchberger]
    This algorithm takes as input a finite set of rules $B^0\subset (F\times F)^\succ$, and outputs a Gr\"obner basis for the monoidal relation $R$ generated by $B^0$. Suppose we have defined a finite set of rules $B^n=\{r_i=(a_i,b_i)\}_{i\in I}\subset (F\times F)^\succ$ for some $n\in\NN$. For each $i,j$ in $I$, let $r_{ij}\in R$ be the relation $(b_i+(a_i\vee a_j)-a_i, b_j+(a_i\vee a_j)-a_j)$ (``eliminate the initial terms of $r_i$ and $r_j$"). As long as a rule in $B^n$ can be applied to a term of $r_{ij}$, apply it and replace that term by the result. This makes one of the terms of $r_{ij}$ decrease and does not change the other, so the process must terminate. If all the resulting $r_{ij}$'s are trivial (i.e. lie in the diagonal of $F\times F$), stop there and output $B^n$. Otherwise, define $B^{n+1}$ by adjoining to $B^n$ the nontrivial $r_{ij}$'s.
\end{algorithm}

\begin{proof}
    For each $n$ such that $B^n$ and $B^{n+1}$ are defined, the ideal $\ini(B^{n+1})$ strictly contains $\ini(B^n)$. Thus, by \ref{lemma:ACC_for_free_monoids}, the algorithm terminates.
    
    The output is a set of rules $B=\{(a_i,b_i)\}_{i\in I}$ such that for each $i,j$ in $I$, the ``term elimination" $r_{ij}:=(b_i+(a_i\vee a_j)-a_i, b_j+(a_i\vee a_j)-a_j)$ can be rewritten as a trivial relation using rules in $B$. We must show that $B$ is a Gr\"obner basis, i.e. that the $a_i$ generate $\ini(R)$. Suppose they do not, and seek a contradiction. Pick some $r=(c,d)$ in $R$ with $c\succ d$, such that $c\notin\ini(B)+F$. Since $B$ generates $R$, there is a finite rewriting sequence
    \begin{equation}\label{rewriting_seq}
        c_0:=c \leftrightarrow_B c_1 \leftrightarrow_B c_2\leftrightarrow_B\dots \leftrightarrow_B c_k:=d.
    \end{equation}
    We may assume that both $r$ and \ref{rewriting_seq} have been chosen to minimize the maximum $m$ of the $c_i$ with respect to $\preceq$. We now claim that we can find a new rewriting sequence from $c$ to $d$ which only involves terms $\preceq m$, and involves $m$ strictly less times than \ref{rewriting_seq} does. By induction, this contradicts the minimality of $m$ and concludes the proof. We will now prove the claim.
    
    Let $1\leq l\leq k$ be such that $c_l=m$. Then, we have $l\neq 1$ and $l\neq k$, and $c_l$ can be rewritten as $c_{l-1}$ and as $c_{l+1}$ using rules of $B$, say $a_i \to b_i$ and $a_j \to b_j$ respectively. If $i=j$, then $c_{l-1}=c_{l+1}$ and we may delete $c_l$ from the rewriting sequence altogether. Thus, we may assume $i\neq j$. Then, $c_l=m$ is in both $a_i+F$ and $a_j+F$, so there is some $d$ such that $m=(a_i\vee a_j)+d$. Therefore, we may write
    \begin{align*}
        c_{l-1} & = d+b_i+(a_i\vee a_j)-a_i \\
        c_{l+1} & = d+b_j+(a_i\vee a_j)-a_j,
    \end{align*}
    i.e. $(c_{l-1},c_{l+1})=(d,d)+r_{ij}$. Since $r_{ij}$ can be rewritten as a trivial relation using rules in $B$, there exists a rewriting sequence with rules of $B$ from $c_{l-1}$ to $c_{l+1}$ which only involves terms smaller than $\max(c_{l-1},c_{l+1})$. We may now substitute the subsequence $c_{l-1} \leftarrow_B c_l \to_B c_{l+1}$ of \ref{rewriting_seq} with this sequence, which proves the claim since $\max(c_{l-1},c_{l+1})\prec m$.
    \end{proof}

\subsection{Localization and integralization}

Here, we show that pure difference ideals are preserved by localization with respect to monomial ideals and by Zariski closure with respect to such localizations. We interpret the integralization of a monoid in terms of localization and closure.

\begin{definition}
    Let $M$ be a monoid. We define the \emph{interior} of $M$, denoted $M^\circ$, to be the ideal of $M$ consisting of those $x\in M$ which do not belong to any face except $M$ itself.
\end{definition}

\begin{lemma}(localization)\label{lemma:localization}
    In \ref{situation:presentation_of_monoid}, let $x$ be an element of $F$. Put $F'=F\oplus \NN t$, let $R'\subset F'\times F'$ be the monoidal relation generated by $R$ and by $(t+x,0)$, and call $M'$ the quotient $F'/R'$. Then $M \to M'$ is initial among monoid maps $M \to N$ such that $x$ is invertible in $N$. In particular, up to a unique isomorphism $M'$ only depends on $x$ via the smallest face of $M$ containing $x$. If $x\in M^\circ$, then $M'=M^\gp$.
\end{lemma}

\begin{proof}
    We have constructed $M'$ by formally adjoining to $M$ an inverse of $x$.
\end{proof}

\begin{lemma}[Eliminating a variable]\label{lemma:eliminating_a_variable}
    In \ref{situation:presentation_of_monoid}, let $t$ be one of the basis elements of the free monoid $F$ and write $F=F_0\oplus \bb Nt$. Assume $t\succ F_0$. If $B$ is a Gr\"obner basis for $R$, then $B_0:=B\cap (F_0\times F_0)$ is a Gr\"obner basis for $R_0:=R\cap (F_0 \times F_0)$.
\end{lemma}

\begin{proof}
    It suffices to show that $\ini(R_0)\subset \ini(B_0)$. Pick any $f\in \ini(R_0)$. Since $B$ is a Gr\"obner basis for $R$, there is a relation in $B$ with initial term $f$. Since $t\succ F_0$, this relation must be in $B_0$.
\end{proof}

\begin{definition}\label{definition:quotient_binomial_ideal}
    Let $F$ be a finitely generated free monoid and $R\subset F\times F$ be a monoidal relation. Let $x$ be an element of $F$. We denote by $(R:x^\infty)$ the set of those $(a,b)\in F\times F$ such that $(a+nx,b+nx)$ is in $R$ for some positive integer $n$.
\end{definition}

\begin{lemma}\label{lemma:integralization_is_an_ideal_quotient}
    With the notations of \ref{definition:quotient_binomial_ideal}, let $M=F/R$ and suppose $x\in F^\circ$. Then, the integralization $M^\int$ of $M$ is the quotient $F/(R:x^\infty)$.\footnote{This lemma would still hold, with an identical proof, if we asked that $F$ be integral instead of free.}
\end{lemma}
\begin{proof}
    The relation cutting out $M^\int$ in $F$ is
    \[
    R^{\int} := \{(a,b) \in F \times F: \exists c \in F:(a+c,b+c) \in R\},
    \]
    which contains $(R:x^\infty)$. Conversely, we must show that an arbitrary $(a,b) \in R^\int$ is in $(R:x^\infty)$. Let $c \in F$ be such that $(a+c,b+c) \in R$. Since $x$ is in $F^\circ$, there exists an integer $n \geq 0$ such that $nx \in c+F$. Since $R$ is monoidal, this implies $(a+nx,b+nx) \in R$, so $(a,b) \in (R:x^\infty)$ as claimed.
\end{proof}

\begin{lemma}\label{lemma:quotient_monoid_cuts_out_zariski_closure_of_localization}
    Under the hypotheses of \ref{lemma:localization}, the map $\Spec \bb Z[F/(R:x^\infty)] \to \Spec \bb Z[M]$ is the Zariski closure of the open immersion $\Spec \bb Z[M']=\Spec \bb Z[M][x^{-1}] \to \Spec \bb Z[M]$.
\end{lemma}

\begin{proof}
    Two elements $a,b$ of $M$ have the same image in $M'$ if and only if $a+nx=b+nx$ for some $n\in \bb N$, so $F/(R:x^\infty)$ is the smallest quotient of $M$ through which $M \to M'$ factors, and the same remains true after applying $\bb Z[-]$.
\end{proof}

\begin{proposition}[see \cite{EisenbudSturmfels}, Corollary 1.7]\label{proposition:quotient_relation_computable_by_eliminating_variable}
    In \ref{situation:presentation_of_monoid}, let $x$ be an element of $F$. Put $F'=F\oplus \NN t$ and let $R'\subset F'\times F'$ be the monoidal relation generated by $R$ and by $(t+x,0)$. Then $(R:x^\infty)=R'\cap(F \times F)$.
\end{proposition}

\begin{proof}
   Since $x\in F'$ is invertible modulo $R'$, we have $(R:x^\infty)\subset R'\cap(F \times F)$. Conversely, pick $(a,b)$ in $R'\cap(F \times F)$. We want to show that $(a,b)+\NN x$ meets $R$. Let $S$ be the set of finite sequences $(a_1,\dots,a_n)$ in $F'$ such that
   \begin{itemize}
       \item $(a_1,a_n)\in (a,b)+\NN x$;
       \item for each $1\leq i<n$, either $a_{i+1}$ and $a_i$ are $R$-equivalent, or one is obtained from the other by applying the rule $t+x \to 0$.
   \end{itemize}
   Since $a$ and $b$ are $R'$-equivalent, $S$ is nonempty. Pick a sequence $(a_1,\dots,a_n)\in S$ with $n$ minimal. We claim that $n=2$, from which the proposition would follow. We will now prove the claim.
   
   Suppose $n>2$. We have $a_1,a_n\in F$ and since $n>2$ is minimal, not all the $a_i$ are in $F$. Therefore, there are indices $i<j$ such that
   \begin{itemize}
        \item $a_{i+1}=a_i+t+x$;
        \item $a_{j+1}+t+x=a_j$;
        \item All the $a_\ell$ with $i<\ell<j$ are $R$-equivalent.
   \end{itemize}
   It follows that $a_i+x$ and $a_{j+1}+x$ are $R$-equivalent. The sequence
   \[
   (a_0+x,\dots,a_i+x,a_{j+1}+x,\dots,a_n+x)
   \]
   is in $S$, thereby contradicting the minimality of $n$.
\end{proof}

\begin{corollary}\label{corollary:algo_computes_integralization} In the setting of \ref{proposition:quotient_relation_computable_by_eliminating_variable}, extend $\preceq$ to $F'$ by declaring $t\succ F$. If $B'$ is a Gr\"obner basis for $\left(R,(t+x,0)\right)\subset F'\times F'$, then by \ref{lemma:eliminating_a_variable} $B=B'\cap(F\times F)$ is a Gr\"obner basis for $(R\colon x^\infty)$. In particular, if $x \in F^\circ$, then $B$ is a Gr\"obner basis for $M^\int$.
\end{corollary}

\subsection{Primes and faces}\label{subsection:primes_and_faces}

Consider a surjective monoid map $\Phi \colon \bb N^n \to M$. By \ref{proposition:integralization_of_faces_dominate_a_logsch} and since integral log schemes are much better understood than non-integral ones, we have an interest in computing the integralizations of the faces of $M$. We have seen how to compute integralizations via Gr\"obner bases and interpret them geometrically (\ref{lemma:localization} and \ref{lemma:quotient_monoid_cuts_out_zariski_closure_of_localization}). We will now give a procedure to compute faces and their Gr\"obner bases.

Let $Z\subset \bb A^n$ be the pure difference scheme corresponding to $\Phi$. A face of $M$ corresponds to a restriction $Z_H \to H$, where $H\subset \bb A^n$ is a linear subspace cut out by a subset $J$ of the coordinates such that the ideal $JM\subset M$ is prime. If $JM$ is not prime, then $Z_H$ is not a pure difference scheme: some of the equations defining $Z$ may become monomial when we restrict to $H$.

\begin{lemma}\label{lemma:basis_elements_generate_primes_and_faces}
    Let $\phi \colon F \to M$ be a surjective monoid map, where $F$ is free with finite basis $E$. Let $p\subset M$ be a prime ideal, and $N$ be the face $M\backslash p$. Call $J$ the finite set of basis elements $E \cap \phi^{-1}p$. Then, $p$ is generated (as an ideal) by $\phi(J)$, and $N$ is generated (as a monoid) by $\phi(E\backslash J)$.
\end{lemma}

\begin{proof}
    Any $f\in \phi^{-1}p$ is a sum of elements of $E$. One of those basis elements is in $\phi^{-1}p$ by primality, so $\phi(J)$ generates $p$. Likewise, any $g\in N$ is a sum of images of basis elements, all of which are in $N$ since $N$ is a face, so $\phi(E\backslash J)$ generates $N$.
\end{proof}

\begin{lemma}\label{lemma:primality_of_a_monoid_ideal_can_be_checked_on_a_presentation}
    Let $\phi\colon F \to M$ and $E$ be as in \ref{lemma:basis_elements_generate_primes_and_faces}. Write $E^\times:= E \cap \phi^{-1}(M^\times)$. Notice that if $p\subset M$ is prime, then $\phi^{-1}p$ may not meet $E^\times$. Let $R\subset F\times F$ be the relation cutting out $M$, and $B$ be a set of generators for $R$. For $J \subset E\backslash E^\times$, the following conditions are equivalent.
    \begin{enumerate}
        \item $J$ is of the form $E\cap \phi^{-1}p$ for some prime ideal $p\subset M$.
        \item For any $(a,b)\in R$, we have $a\in J+F$ if and only if $b\in J+F$.
        \item For any $(a,b)\in B$, we have $a\in J+F$ if and only if $b\in J+F$.
    \end{enumerate}
\end{lemma}

\begin{proof}
    The relation $\sim$ on $F\times F$ defined by $a\sim b$ whenever [$a\in J+F$ if and only if $b\in J+F$] is symmetric, reflexive, transitive and monoidal, so it contains $B$ if and only if it contains $R$ and $(3) \Longleftrightarrow (2)$.

    We now prove $(1) \Longleftrightarrow (2)$. Let $p\subset M$ be the ideal 
    \[
    p=\phi(J+F)=\phi(J)+M.
    \]
    Since $R$ is the equivalence relation given by the surjection $F \to M$, $(2)$ is equivalent to the ideal $J+F\subset F$ being the preimage of a subset of $M$, which is necessarily $p$. We claim that $p$ is prime if $(2)$ holds. Indeed, let $(a,b,c)\in F\times F \times (J+F)$ such that $\phi(a)+\phi(b)=\phi(c)$. Then $(a+b,c)\in R$ so $a+b\in J+F$ by $(2)$. Since $J+F$ is a prime ideal of $F$, this implies that either $a$ or $b$ is in $J+F$. Hence, $p$ is prime as claimed, so we have shown that $(2)$ holds if and only if $J+F$ is of the form $\phi^{-1}q$ for some \emph{prime} ideal $q\subset M$, which is equivalent to $(1)$ by \ref{lemma:basis_elements_generate_primes_and_faces}.
\end{proof}

\begin{lemma}
    Let $\phi\colon F \to M$, $p$, $N$, $E$ and $J$ be as in \ref{lemma:basis_elements_generate_primes_and_faces}, so that $\phi^{-1}N$ is the face $F^0=\bigoplus\limits_{a\in E\backslash J} \bb N a$ of $F$. Call $R$ the relation in $F\times F$ cutting out $M$. Then $N$ is cut out in $F^0$ by $R\cap (F^0 \times F^0)$, and the square of pointed monoids
    \[
    \begin{tikzcd}
        F \cup\{\infty\} \arrow[r]\arrow[d] & M \cup\{\infty\} \arrow[d] \\
    F^0 \cup \{\infty\} \arrow[r] & N \cup \{\infty\}
    \end{tikzcd}
    \]
    is cocartesian.
\end{lemma}

\begin{proof}
    Existence and cocartesianness of the square follow from \ref{lemma:quotient_of_pointed_monoid_by_prime_ideal_is_the_corresponding_face}. Any $a,b \in F^0$ have the same image in $N$ if and only if they have the same image in $M$, so $F^0 \to N$ is cut out by $R\cap (F^0 \times F^0)$.
\end{proof}

\subsection{Saturation}

Throughout this subsection, $M$ is an integral and finitely generated monoid with a given finite presentation
\[
\bigoplus_{i\in\ca I} \bb N.e_i/(a_j=b_j)_{j\in \ca J},
\]
and our goal is to compute $M^\sat$.

\begin{remark}
    Computing $M^{sat}$ is analogous to computing the integral closure of a reduced ring inside its fraction field. In fact, when $M$ is finitely generated it follows from \cite[Corollary 12.6]{GilmerCommutativeSemigroupRings} that the integral closure of $k[M]$ is $k[M^{sat}]$ for any field $k$. One could compute this integral closure using the general (but costly) algorithm of \cite{DeJongAnAlgoForComputingIntegralClosure}, which de Jong attributes to Grauert and Remmert \cite{GrauertRemmertAnalytische}, \cite{GrauertRemmertCoherentAnalyticSheaves}. We present a method better suited to compute saturations.
\end{remark}

Let $A$ be the matrix obtained by expressing the vectors $a_j-b_j$ in the basis $(e_i)_{i \in I}$. We may compute the Smith normal form of $A$ in polynomial time, thereby obtaining a presentation for the torsion subgroup $M^\tor$ of $M^\gp$ and an explicit isomorphism $M^\gp \iso \bb Z^n \oplus M^\tor$. For each $i \in I$, let $v_i$ be the image of $e_i$ in $\bb Z^n$. The image $N$ of $M \to \bb Z^n$ is the submonoid spanned by the $v_i$, and $M^\sat = N^\sat \oplus M^\tor$, so it suffices to compute $N^\sat$.

The saturation $N^\sat$ is the monoid of all integral points in the cone $\bb Q_{\geq 0}N := N \otimes_{\bb N} \bb Q_{\geq 0}$ spanned by the $v_i$ in $\bb Q^n$. Computing a presentation for $N^\sat$ as a monoid is costly, so it is convenient to store $N^\sat$ as the data of the $v_i$ (or any other generating set for $\bb Q_{\geq 0}N$) when possible. If one really needs a presentation, one can compute it as follows.

\begin{enumerate}
    \item Write $\bb Q_{\geq 0}N$ as a union of strongly convex cones with $n$ rays, thereby reducing to the case where $\bb Q_{\geq 0}N$ has $n$ rays with primitive vectors $v_1,...,v_n\in\bb Z^n$.
    \item Find the set $\Lambda^0$ of all lattice points in the semi-open hypercube $\prod_{1\leq j\leq n}[0,1[v_j$. This can be done e.g. by computing a LLL basis $(e'_1,...,e'_n)$ \cite{lenstra1982factoring} with respect to the Euclidean norm $\|-\|\colon\sum_{j}\lambda_jv_j \mapsto \sum_j \lambda_j^2$ and probing the hypercube with translates of the fundamental domain $\prod_{1\leq j\leq n}[0,1[e'_j$.
    \item $N^{sat}$ is now the submonoid of $\bb Z^n$ generated by $\Lambda^0\sqcup\{v_1,...,v_n\}$.
\end{enumerate}

\begin{remark}
    This method has exponential worst-case complexity and might seem rough at first glance, but it is analogous to the best known algorithms for the short vector problem, which is known to be NP-hard. The original proof of NP-hardness is \cite{van1981another}. For a modern survey on related problems and results, we refer to \cite{BennettTheComplexityOfTheShortestVectorProblem}.
\end{remark}

\subsection{Algorithms}

We summarize the results of \ref{section:algorithms_on_monoids} in the form of a list of algorithms. In what follows, $E$ is a finite set, $F$ is the free monoid $\bigoplus\limits_{e\in E} \bb N e$, $R$ is a monoidal relation on $F$ and $\phi \colon F \to F/R=:M$ is the corresponding surjective monoid map.

\begin{algorithm}[Invertible basis elements]\label{algorithm:invertible_basis_elements}

Input: A monomial order $\preceq$ on $F$ and a Gr\"obner basis $B\subset F\times F$ for $R$ with respect to $\preceq$. \\
    Output: The set $E^\times:= E\cap\phi^{-1} M^\times$.\\
    Procedure: For each $e\in E$ rewrite $e$ as much as possible using rules of $B$. Since $B$ is a Gr\"obner basis, the result is the normal form of $e$ with respect to $\preceq, R$. The subset of $E$ made up of elements whose normal form is $0$ is $E^\times$.
\end{algorithm}

\begin{algorithm}[Prime ideals]\label{alg-for-prime-ideals} Input: a finite generating set $B$ for $R$, and the subset $E^\times =E\cap\phi^{-1}M^\times$ of $E$.\\
    Output: All the subsets of $E$ of the form $E\cap \phi^{-1}p$, where $p$ is a prime ideal of $M$. Those are in explicit bijection with primes of $M$ and with faces of $M$ by \ref{lemma:basis_elements_generate_primes_and_faces}.\\
    Procedure: Output the subsets of $E\backslash E^\times$ which satisfy condition (3) of \ref{lemma:primality_of_a_monoid_ideal_can_be_checked_on_a_presentation}.
\end{algorithm}

\begin{algorithm}\label{algorithm:computation_integralized_faces}
    Input: A subset $E'\subset E$ such that the image $M'$ of $F':=\bigoplus\limits_{e\in E'} \bb N e$ in $M$ is a face, and a Gr\"obner basis $B$ of $R$ with respect to a monomial order $\preceq$ such that $f\succ F'$ for all $f \in F\backslash F'$. \\
    Output: A $\preceq$-Gr\"obner basis for the relation $(R')^\int$ cutting out $(M')^\int$ inside $F'$.\\
    Procedure: Put $x=\sum\limits_{e\in E'} e$. Restrict $B$ to $F'\times F'$ to obtain a Gr\"obner basis for the relation $R'$ cutting out $M'$. Then, compute a Gr\"obner basis for $(R')^\int=(R'\colon x^\infty)$ using \ref{proposition:quotient_relation_computable_by_eliminating_variable}. Equivalently, one could first compute a Gr\"obner basis for $(R\colon x^\infty)$ using \ref{proposition:quotient_relation_computable_by_eliminating_variable} then restrict to $F' \times F'$.
\end{algorithm}

\begin{remark}
    If one is interested in computing all the integralized faces simultaneously rather than just one of them, it is suboptimal to compute Gr\"obner bases for all lexicographic monoidal orders and apply \ref{algorithm:computation_integralized_faces} to all faces. The improvements one can make include
    \begin{itemize}
        \item Computing Gr\"obner bases only with respect to the lexicographic monoidal orders in a family $(\preceq_i)_{i\in I}$ such that for each face $M'\subset M$, there is at least one $i\in I$ such that $\phi^{-1}M' \preceq_i F\backslash \phi^{-1}M'$.
        \item Rather than applying \ref{algorithm:computation_integralized_faces} to each face separately, avoid redundant computations using the fact that if $x=x'+x''$ then $(R':x^\infty)=((R':x'^{\infty}):x''^\infty)$.
    \end{itemize}
\end{remark}

\section{A worked example}\label{section:worked_examples}

We present a basic example of a coherent, non-integral logarithmic stack and how our algorithms work in a simple case. We study a double ramification cycle geometry, but the picture for stable maps is similar.

As we explain in~\ref{section: examples}, fiber products of double ramification cycles produce geometrically meaningful examples of non-coherent logarithmic stacks. As noted there, describing the fine faces of these stacks essentially boils down to pulling back double ramification cycles along blowups of $\overline{\mathcal M}_{g,n}$. The fine faces of these pullbacks can then be described using our algorithms, and then intersected with other such pullbacks to gain a description of the full fiber product. 

The following example gives a sense for how this computation looks, and how the algorithms developed here can be applied. 

\subsection{The double ramification cycle under blowups} 

We consider the double ramification cycle space $\mathsf{DR}_1(2,-2)$. This is a moduli space of dimension $1$. Generically, it parameterizes two-pointed genus $1$ curves $(E,p,q)$ such that $2p\sim 2q$. In the boundary, it can be given a moduli interpretation by using logarithmic elliptic curves~\cite{Molcho_Wise_2022}. See also~\cite{DeligneRapoport} for a similar discussion.

The space $\mathsf{DR}_1(2,-2)$ admits a morphism
\[
\mathsf{DR}_1(2,-2)\to \overline{\mathcal M}_{1,2}.
\]
The moduli space $\overline{\mathcal M}_{1,2}$ has two boundary curves and two $0$-dimensional strata. One of the boundary curves, called $\Delta_0$, is the closure of the locus where the elliptic curve becomes self-nodal. The other one, called $\Delta_1$, is the closure of the locus where the two marked points come together.

The space $\mathsf{DR}_1(2,-2)$ contains $\Delta_1$. The curve $\Delta_0$ is itself, as a space, a self-nodal rational curve. The node $x \in \Delta_0$ is a $0$-dimensional stratum, parametrizing the stable curve with rational components whose dual graph is depicted in \ref{fig:dualgraph_x}.

\begin{figure}
    \centering
    \begin{tikzpicture}[>=stealth, node distance=3cm,
  every node/.style={circle, draw, minimum size=0.8cm},
  bend angle=30
]

\node (p) {$p$};
\node (q) [right of=p] {$q$};

\draw[bend left] (p) to (q);
\draw[bend left] (q) to (p);

\draw[->] (p) -- ++(150:1.5cm); 
\draw[->] (q) -- ++(30:1.5cm);  

\end{tikzpicture}
    \caption{The dual graph of the curve parametrized by the point $x$. The arrows with no endpoints correspond to markings.}
    \label{fig:dualgraph_x}
\end{figure}

\'Etale locally near $x$, the space $\overline{\mathcal M}_{1,2}$ is isomorphic to $\mathbb A^2$, with the axes corresponding to the branches of $\Delta_0$ at $x$. The coordinates on this $\mathbb A^2$ are canonical up to exchange and multiplication by a scalar. If we call the coordinates $u$ and $v$, then the curve  $\mathsf{DR}_1(2,-2)$ restricts to one of the form $u = \lambda v$, for some scalar $\lambda$. This follows, for example, by the description of deformation spaces of admissible covers due to Harris--Mumford~\cite{HM82}; the moduli problem is slightly different here but the local analysis is identical.

We form a new space by blowing up $\overline{\mathcal M}_{1,2}$ at $x$ to form $\widetilde{\mathcal M}$ and pulling back to form:
\begin{equation}\label{afewagewgreatrea}
    \begin{tikzcd}
        \mathbb D\arrow{r}\arrow{d} & \mathsf{DR}_1(2,-2) \arrow{d}\\
        \widetilde{\mathcal M}\arrow{r} & \overline{\mathcal M}_{1,2}.
    \end{tikzcd}
\end{equation}
Scheme theoretically, the space $\mathbb D$ consists of the union of $\mathsf{DR}_1(2,-2)$ with a $\mathbb P^1$.

The space $\mathbb D$ is equipped with a coherent logarithmic structure that is not fine, making the square \ref{afewagewgreatrea} cartesian in the category of not-necessarily-fs log schemes. This log structure is the pushout of the ones induced by $\widetilde{\mathcal M}, \mathsf{DR}_1(2,-2), \overline{\mathcal M}_{1,2}$.

\'Etale locally near $x \in \overline{\mathcal M}_{1,2}$ we can find two elements in the logarithmic structure, whose images, say $a$ and $b$, in the characteristic monoid freely generate it. We view the parameters $a$ and $b$ as being edge lengths for the dual graph of the curve at $x$ (\ref{fig:dualgraph_x}).

The fibre of $\mathsf{DR}_1(2,-2)$ above $x$ is a single point, at which the logarithmic structure has rank $1$. The monoid map
\[
\mathbb Na \oplus \mathbb Nb \to \mathbb N
\]
dual to the map $\mathsf{DR}_1(2,-2)\to \overline{\mathcal M}_{1,2}$ takes $a$ and $b$ to $1$.

Now on the blowup $\widetilde{\mathcal M}$. There are two charts for the logarithmic structure -- exactly corresponding to the two standard charts on the exceptional $\mathbb P^1$. On one of the charts, we introduce a third generator $c$ and the relation
\[
a+c = b,
\]
and on the other chart we introduce $b+c = a$. 

Combining these, over this chart the characteristic monoid for $\mathbb D$ is
\begin{equation}\label{eqn:lalaagrew}
    \o M := (\bb Na \oplus \bb Nc)/(a+c=a)
\end{equation}
To spell this out, every element in this monoid can be written as $\lambda a+ \mu c$. Using the relation, if $\lambda$ is nonzero, the $\mu c$ term can be eliminated. But if $\lambda$ is $0$, then there is an $\mathbb N$-worth of elements. This monoid is therefore not fine. 

In order to run the algorithm on $\o M$, we need to input generators and relations. We choose the generators $a$ and $c$ and the same relation
\[
a+c = a.
\]
We now compute the prime ideals, or equivalently, the faces of $\o M$ using~\ref{alg-for-prime-ideals}. This uses~\ref{lemma:primality_of_a_monoid_ideal_can_be_checked_on_a_presentation}, which tells us that every prime ideal is generated by a subset of the basis $a,c$ for the ambient free monoid $F=\bb Na \oplus \bb Nc$ and in order to check if a given subset generates a prime ideal, we simply need to check the conditions in the lemma. 

In our case, there are four subsets to check -- $\emptyset$, $\{a\}$, $\{c\}$, and $\{a,c\}$. Let us consider the subset $\{c\}$. Since $a+c$ is in the ideal of $F$ generated by $c$ but $a$ is not, the condition fails: $c$ does not generate a prime ideal, equivalently $a$ does not generate a face. The other three subsets do correspond to prime ideals, so their complements correspond to faces. One can now check directly that the two irreducible components of $\mathbb D$ are fine faces, corresponding to the face inclusions
\begin{align*}
    \bb Nc & \subset \o M \\
    \o M & \subset \o M
\end{align*}

\subsection{Tropical perspective} One way to interpret the preceding discussion is via extended tropicalizations. 

In the previous example the edge lengths were $a$ and $b$. If we allow these lengths the possibility of taking the value $\infty$, then the possible set of values is $(\mathbb N\cup \{\infty\})^2$. Similarly, we can allow the length parameter $a=b$ on $\mathsf{DR}_1(2,-2)$ near $x$ to take infinite values. Finally, on one chart of the blowup we have parameters $a$ and $c$, viewing $b$ as $a+c$. Both $a$ and $c$ take values in $\mathbb N\cup \{\infty\}$. From the point of view of tropical curves, we allow the edge lengths $a$ and $b$ to be infinite. In the space $\mathbb D$, we simply record their difference $c$. If $a$ and $b$ are finite, $c$ must be $0$, but if $a$ and $b$ are infinite, $c$ can be anything. 

The picture in general for calculating the fine faces of fiber products of $\mathsf{DR}_g(A_i)$, posed in~\ref{section: examples} can be understood in similar terms. On each cone of the tropical moduli space $\mathcal M_{g,n}^{\sf trop}$, there are natural subcones of curves that satisfy edge length constraints coming from $\mathsf{DR}_g(A_i)$ -- in our toy case, it was the condition that the two edge lengths were equal. Given a map of such cones, say $\delta\to \sigma$, we can pass to extended tropicalizations to obtain a map
\[
\overline\delta\to\overline \sigma.
\]
Similarly, if $\sigma'\to \sigma$ is a subdivision, we get a map of extended fans
\[
\overline \sigma'\to\overline \sigma.
\]
These extended tropicalizations can be thought of in the terms described above -- tropical curves with edge length conditions, possibly including the information of some edge lengths being longer than others (to pass to charts) and the differences between edge lengths when this is true. 

The fiber product of extended tropicalizations therefore has faces, exactly corresponding to the faces of the coherent pullback, which can be viewed as tropical curves with additional information about finite length differences between possibly infinite length parameters.

\bibliographystyle{alpha} 
\bibliography{biblio}

\end{document}